\documentclass[journal]{IEEEtran}
\usepackage{amsmath, amsfonts, amsthm, amssymb}
\usepackage{algorithmic}
\usepackage{algorithm}
\usepackage{array}
\usepackage[caption=false,font=footnotesize]{subfig}
\usepackage{textcomp}
\usepackage{stfloats}
\usepackage{url}
\usepackage{verbatim}
\usepackage{graphicx}

\usepackage{cite}
\usepackage{mathtools}
\usepackage{booktabs} 
\usepackage{siunitx}   
\usepackage{makecell}
\usepackage{threeparttable}
\usepackage{microtype}
\usepackage{multirow}
\usepackage{placeins}
\usepackage[colorlinks=true, linkcolor=black, citecolor=black, urlcolor=black]{hyperref}

\newcommand{\prob}[2][]{\text{\bf P}\ifthenelse{\not\equal{}{#1}}{_{#1}}{}\!\left(#2\right)}
\newcommand{\expect}[2][]{\text{\bf E}\ifthenelse{\not\equal{}{#1}}{_{#1}}{}\!\left[#2\right]}
\newcommand{\var}[2][]{\text{\bf Var}\ifthenelse{\not\equal{}{#1}}{_{#1}}{}\!\left[#2\right]}

\newcommand{\given}{\, : \,}
\newcommand{\sB}{\mathcal{B}}

\newcommand{\sE}{\mathcal{E}}      
      
\newcommand{\sG}{\mathcal{G}}      
      
\newcommand{\sI}{\mathcal{I}}

\newcommand{\sN}{\mathcal{N}}

\newcommand{\sT}{\mathcal{T}}

\newcommand{\R}{\mathbb{R}}                     % Reals.
\newtheorem{theorem}{Theorem}
\newtheorem{remark}[theorem]{Remark}

\newtheorem{lemma}[theorem]{Lemma}

\newtheorem{proposition}[theorem]{Proposition}
\newtheorem{definition}[theorem]{Definition}

\newtheorem{assumption}[theorem]{Assumption}

\newcommand{\RR}{\mathbb{R}}

\newcommand{\norm}[1]{\left\| #1 \right\|}

\begin{document}

% \title{Assessment of Dynamic Line Ratings in ACOPF}
\title{Dynamic Line Ratings in AC Optimal Power Flow: Transient Temperature, Decomposition, 
and Large-scale Evaluation}

\author{Baptiste Rabecq,~\IEEEmembership{Student Member,~IEEE}, Thomas Lee,~\IEEEmembership{Student Member,~IEEE}, Andy Sun,~\IEEEmembership{Senior Member,~IEEE.}
        % <-this % stops a space
        \thanks{Baptiste Rabecq is with the Operations Research Center, MIT}
        \thanks{Thomas Lee is with the Institute for Data, Systems, and Society, MIT}
        \thanks{Andy Sun is with the Sloan School of Management, MIT}
% \thanks{This paper was produced by the IEEE Publication Technology Group. They are in Piscataway, NJ.}% <-this % stops a space
% \thanks{Manuscript received April 19, 2021; revised August 16, 2021.}
}

% The paper headers
% \markboth{IEEE T \LaTeX\ Class Files,~Vol.~14, No.~8, August~2021}%
% {Shell \MakeLowercase{\textit{et al.}}: A Sample Article Using IEEEtran.cls for IEEE Journals}

% \IEEEpubid{0000--0000/00\$00.00~\copyright~2021 IEEE}
% Remember, if you use this you must call \IEEEpubidadjcol in the second
% column for its text to clear the IEEEpubid mark.

\maketitle
\begin{abstract}
As power grids experience increasing renewable penetration and rapid load growth from AI data centers and electrification, alleviating line congestion becomes critical to unlocking additional grid capacity. This work investigates Dynamic Line Rating (DLR), a congestion mitigation method that adjusts power line current limits in response to meteorological conditions. Unlike traditional approaches that impose predefined time-varying limits, we propose a novel optimization framework that embeds the transient-state heat equation governing conductor temperature dynamics, enabling direct constraints on conductor temperature rather than simplified steady-state approximations. 
We derive a closed-form solution to the heat equation, enabling a finite-dimensional reformulation of the dynamics. We then leverage a distributed decomposition method, a bi-level Alternating Direction Method of Multipliers (ADMM) algorithm with provable convergence, aided by regularity properties of the heat equation solution. These modeling and algorithmic innovations allow us to conduct the first large-scale evaluation of DLR using multi-period AC optimal power flow. Numerical experiments on the 2000-bus Texas grid demonstrate that DLR allows significant reduction in generation cost in congested systems over Static Line Rating (SLR) and Ambient Adjusted Ratings (AAR). The transient temperature formulation provides additional grid flexibility and headroom benefits with minimal computational overhead.
\end{abstract}

\section{Introduction}
\begin{table}[t]
\setlength{\tabcolsep}{2pt}
\caption{Notation table for sets, variables, and parameters}
\label{tab:notations}
\centering
\begin{tabular}{|c|p{0.56\columnwidth}|}
\hline
\textbf{Notation} & \textbf{Description} \\
\hline
\multicolumn{2}{|c|}{\textbf{Sets}} \\
\hline
$\sI, \sN_i$ & Set of $n$ buses, neighbor buses of bus $i$ \\
$\sB$, $\sE \subseteq \sB$ & Set of branches, lines. \\
% $Trans$ & Set of transformers \\
$\sG$, $\sG_i$ & Set of generators, generators at bus $i$ \\
$\sT$, $\sT'$ & Decision times: $[M]{:=}\{1,\dots,M\}$, $\sT\!\setminus\!\{1\}$ \\
\hline
\multicolumn{2}{|c|}{\textbf{Optimization Variables}} \\
\hline
$p^G_{g,t}$, $q^G_{g,t}$ & Real/complex power from gen. $g$ at time $t$ \\
$p_{i,t}$, $q_{i,t}$ & Net real/complex power at bus $i$ at time $t$ \\
$\iota_{ij,t}, I_{ij,t}^{\mathrm{re}}, I_{ij,t}^{\mathrm{im}}$ & Squared magnitude, real and imaginary current on line $ij$ at time $t$ \\
$e_{i,t}$, $f_{i,t}$ & Real/complex voltage at bus $i$, time $t$ \\
$T_{ij}(\tau)$, $T_{ij,t}$ & Continuous/discrete temperature of line $ij$ \\
\hline
\multicolumn{2}{|c|}{\textbf{Electrical Parameters}} \\
\hline
$B_{ij}, G_{ij}$ & Line $ij$ susceptance, conductance \\
$B_i, G_i$ & Self-susceptance/conductance at bus $i$ \\
$Y_{ij}, Y_i$ & (Shunt) admittance of line $ij$, bus $i$ \\
$C_g(\cdot)$ & Generation cost function of generator $g$\\
$V^{\min{/}\max}_i$ & Voltage limits at bus $i$ \\
$P^{\min{/}\max}_g,Q^{\min/\max}_g$ & Power limits for generator $g$ \\
$P^D_{i,t}$, $Q^D_{i,t}$ & Power demand at bus $i$ at time $t$ \\
$\theta^{\min}_{ij}$, $\theta^{\max}_{ij}$ & Angle diff. limit on line $ij$ \\
$P^{Rmp,\pm}_g$ & Generator $g$ ramp rates (up/down) \\
\hline
\multicolumn{2}{|c|}{\textbf{Physical Parameters}} \\
\hline
$T^{\max}$, $T_{0,ij}$ & Max, initial line temperature \\
$\Delta$ & Dispatch interval \\
$r_{ij}$,$m_{ij}$,$L_{ij},c_{p, ij}, D_{0, ij},\gamma_{ij}$ & Line $ij$ resistivity, mass, length, specific heat capacity, diameter, emissivity \\
$T_{a,ij}$, $v_{w,ij}$, $\phi_{ij}$ & Atmospheric temp, wind speed, wind angle \\
\hline
\end{tabular}
\vspace{-2.5em}
\end{table}
\IEEEPARstart{T}he increasing demand for efficient and reliable power transmission has highlighted the limitations of Static Line Ratings (SLR), which impose conservative thermal limits on transmission lines based on worst-case weather assumptions. This leads to underutilization of network capacity by increasing congestion costs and lowering system efficiency. Dynamic Line Ratings (DLR) leverage real-time weather data (ambient temperature, wind speed, and solar radiation) to adjust line ratings dynamically, increasing effective line ampacity. Recognizing its potential, the Federal Energy Regulatory Commission (FERC) has advocated for broader DLR adoption to enhance grid flexibility and reduce operational costs \cite{noauthor_e-1-rm20-16-000_nodate, noauthor_ferc_nodate}. 
However, deploying real-world DLR solutions presents two major challenges. First, it requires high-resolution weather data and/or temperature monitoring technologies. Second, it adds computational complexity to the Optimal Power Flow (OPF) problems used for electricity dispatch. In this work, we assume full access to this data and focus on performance assessment. 

In Optimal Power Flow (OPF) problems, operators minimize generation costs while satisfying system constraints, e.g., voltage levels, angle limits, and transmission capacities. The most accurate formulation, Alternating Current OPF (ACOPF) \cite{carpentier_contribution_1962}, captures nonlinear nonconvex power flow physics but is strongly NP-hard \cite{bienstock_strong_2019}. To reduce complexity, the industry often uses the linear Direct Current (DCOPF) approximation, although its accuracy degrades under heavily loaded conditions.

As a result, most DLR studies to date have relied on DCOPF approximations, demonstrating benefits in capacity expansion \cite{jabarnejad_optimal_2016} and operational dispatch \cite{verma_real_2020} in small to medium-scale systems.
Recent work has extended DLR to contingency-constrained DCOPF \cite{lyu_impacts_2023} and 39-bus ACOPF \cite{su_dynamic_2024}, while \cite{lee_impacts_2022} applied DCOPF-based DLR to a 2000-bus ERCOT model. Comprehensive reviews of DLR advancements can be found in \cite{erdinc_comprehensive_2020, fernandez_review_2016}. A comparison of existing contributions is provided in Table \ref{tab:dlr_lit_review}.
A common limitation across prior work is the reliance on steady-state thermal ratings, which assume line temperatures stabilize over 20–60 minutes \cite{ieee2012ieee}. This assumption breaks down under modern $5$-minute dispatch intervals, where transient thermal dynamics can significantly impact line limits. We aim to go past this limitation to include real-time temperature dynamics in our framework.

\begin{table}[!t]
\centering
\caption{DLR-Based OPF Studies Comparison}
\label{tab:dlr_lit_review}
% \begin{threeparttable}
\begin{tabular}{@{}ccccp{1.5cm}@{}}
\toprule
\textbf{Reference} & \textbf{ACOPF} & \textbf{Ramp} & \textbf{Transient Temp.} & \textbf{Size} \\
\midrule
\cite{glaum_leveraging_2023} & No & No & No & $\approx 100$-bus \\
\cite{jabarnejad_optimal_2016} & No & No & No & 118-bus \\
\cite{verma_real_2020} & Yes & Yes & No & 4-bus \\
\cite{lyu_impacts_2023} & No & No & No & 118-bus \\
\cite{su_dynamic_2024} & Yes (heuristic) & No & No & 39-bus \\
\cite{lee_impacts_2022} & No & No & No & 2000-bus \\
\textbf{This work} & \textbf{Yes} & \textbf{Yes} & \textbf{Yes} & \textbf{2000-bus} \\
\bottomrule
\end{tabular}
% \begin{tablenotes}
% \item[$\dagger$] Heuristic method
% \end{tablenotes}
% \end{threeparttable}
\vspace{-2.5em}
\end{table}

Other works consider temperature-dependent power flows by modeling the effects of variable resistivity and sag \cite{murphy2020temperature}, or by controlling thermal dynamics through stochastic formulations \cite{bienstock2016stochastic}. However, these studies do not address large-scale grid optimization or system-level dispatch. No existing approach integrates DLR with large-scale ACOPF nor captures transient behavior critical to fast-timescale grid operations.

This work introduces a novel framework that integrates transient temperature dynamics into a multi-period ACOPF formulation. By deriving a closed-form solution to the heat equation and applying a time-space decomposition, it captures the thermal behavior of transmission lines within a large-scale network. The problem is solved using a bi-level Alternating Direction Method of Multipliers (ADMM) algorithm \cite{sun_two-level_2023}, with a 3-block nonconvex ADMM at the inner level and iterative feasibility enforcement at the outer level. The approach scales to a $2000$-bus network.
We theoretically derive the convergence rate of our algorithm by analyzing the smoothness of the heat dynamics, reformulated as an ordinary differential equation (ODE), consistent with results in \cite{sun_two-level_2023}.
We then empirically demonstrate that DLR significantly reduces system costs and renewable congestion in AC networks, compared to SLR. Additionally, we show that incorporating transient thermal dynamics provides operational headroom and enhances grid flexibility compared to steady-state formulations. 

To our knowledge, this is the first study to quantify the economic and operational value of DLR on a realistically sized grid while retaining full AC power-flow physics and transient conductor-temperature dynamics.

The paper is structured as follows. In section~\ref{sec:temp_dyn}, we derive a closed-form solution to the transient temperature dynamics, and reformulate them with a finite number of nonlinear equalities. In section~ \ref{sec:problem}, we introduce our optimization model. In section~\ref{sec:decomposition}, we present our decomposition strategy and our ADMM algorithm. In section~\ref{sec:convergence}, we prove its convergence. In section \ref{sec:computations}, we present computational experiments.
% \vspace{-0.1em}
\label{sec:intro}
\section{Study of the Temperature Dynamics}\label{sec:temp_dyn}
\noindent
We now derive a closed-form solution to the temporal evolution of the lines under certain assumptions, and formulate the multi-period temperature model with a finite number of inequalities. 
\subsection{Single-period dynamics}\label{sec:ieee_single}
\noindent
We first study the temperature evolution of a line over a single time interval $[0, \Delta]$, where $\Delta >0$. In our computational study, $\Delta$ is set to 5 minutes. The squared magnitude of current $\iota_{ij}$ on line \(ij\in\sE\)  is assumed to be constant over this time interval. 
% All results in this section are on the interval $[0, \Delta]$.
Assuming each line \(ij\in\sE\) has no radial or longitudinal gradient of temperature, the conductor temperature \(T_{ij}\) (in Kelvins) of a line evolves according to:
\begin{equation}
  m_{ij}c_{p, ij}\frac{dT_{ij}}{d\tau}
  {=} R_{ij}(T_{ij})\iota_{ij}
  {+} q^{s}_{ij}
  {-} q^{c}_{ij}(T_{ij})
  {-} q^{r}_{ij}(T_{ij}) ,
  \label{eq:heat_balance}
\end{equation}
where the radiative loss is $q^{r}_{ij}(T)=\pi D_0\epsilon\sigma\bigl(T^{4}-T_a^{4}\bigr).$
% with \(D_0\) the conductor diameter, \(\epsilon\) its emissivity
% (\(0.2\!-\!0.9\)), and \(\sigma=5.67\times10^{-8}\,\mathrm{W\,m^{-2}K^{-4}}\).
The solar gain \(q^{s}_{ij}\) depends only on irradiance and thus on geographic location and sky clarity, not on \(T\).
Convective cooling is taken as the forced component $q^{c}_{ij}(T) = K_{c, ij}\bigl(T-T_{a, ij}\bigr)$, where $K_{c , ij}$ is a nonlinear function of the conductor temperature, the wind speed at the line $v_{w, ij}$ and the angle $\phi_{ij}$ between the wind and the line $ij$. The details of the expressions can be found in \cite{ieee2012ieee}.
We now make additional assumptions:
%---------------------------------------------------------------------------%
    \begin{assumption}
        The resistivity of the line $R(T_{ij})$ does not depend on the temperature $T_{ij}$. We denote by $r_{ij}$ the resistivity of the line $ij$.\label{ass:resistivity}
    \end{assumption}
    \vspace{-1em}
    \begin{assumption} The weather conditions $(\phi, v_w, T_a)$ are constant over the small time interval $[0, \Delta]$. Therefore, the energy fluxes are functions only of the conductor temperature. \label{ass:constant}
    \end{assumption}
    \vspace{-1em}
    \begin{assumption} For each line, we assume that there exists $k^{0, c}_{ij}, k^c_{ij}$ such that $q^c_{ij}(T) = k^c_{ij, t} \cdot T + k^{0, c}_{ij, t}$. \label{ass:lin-conv} \end{assumption}

    \begin{remark}
        Assumption~\ref{ass:resistivity} simplifies the standard assumption that resistance increases linearly with temperature. Taking $r_{ij}$ as the ``hot" resistance (for $T = T^{\max}$) provides conservative results. Assumption~\ref{ass:constant} is standard and necessary to derive closed-form solutions to \eqref{eq:heat_balance}. Assumption~\ref{ass:lin-conv} approximates the convective term as a linear function of $T$. We observed empirically that the nonlinearity in $K_{c , ij}$ is negligible in the range of temperature we consider. Our linear regression approximation has $R^2 \geq 0.99$ in all weather conditions. 
    \end{remark}
These assumptions allow us to reformulate and solve the differential equation \eqref{eq:heat_balance}.
\begin{theorem}[Closed Form Solution]\label{thm:tempsol}
    Suppose that Assumptions \ref{ass:resistivity}-\ref{ass:lin-conv} hold. Then, for each line $ij$, the heat equation \eqref{eq:heat_balance} can be reformulated as:
    \begin{align}
        \frac{dT_{ij}}{d\tau} = - K_{4 , ij} T_{ij}^4 -K_{1 , ij} T_{ij} {+} K_{0, ij},  \label{eq:temp-dyn-linear-qc}
    \end{align}
    where $K_{0, ij} {=} K_{0 , ij}'{+} r_{ij}'\iota_{ij}$, $r'_{ij}{=} \frac{r_{ij}}{m_{ij}c_{p , ij}}$, $K_{1, ij} {=} \frac{k^c_{ij}}{m_{ij} c_{p , ij}}$, $K_{0 , ij}' {=} \frac{\pi D_{0 , ij} \epsilon \sigma T_{a , ij}^4 {+} q^s_{ij} {-} k^{c,0}_{ij}}{m_{ij} c_{p , ij}}$,  and $K_{4, ij} = \frac{\pi D_{0, ij} \epsilon_{ij} \sigma}{m_{ij} c_{p , ij}}$.\\
    When $K_{4, ij}>0$ and $K_{0, ij}>0$, the quartic polynomial $P_{ij}(T):=  T^4 + \frac{K_{1, ij}}{K_{4, ij}} T - \frac{K_{0, ij}}{K_{4, ij}}$ has two distinct real roots, a positive one denoted by $s_{1, ij}$ and a negative one denoted by $-s_{2, ij}$. Moreover, given an initial temperature $T_{ij ,0}$ at time $\tau=0$, the solution to the ODE \eqref{eq:temp-dyn-linear-qc} is given by (we omit the subscript $ij$):
    {\small
    \begin{align}\label{eq:solT}
       \tau(T) &{=} \begin{aligned}[t]
           &\frac{1}{ K_4} \Bigg[\frac{s_2-s_1}{g_1g_2}\log\frac{|T^2-pT+q|}{|T_0^2-pT_0+q|}\\
           &\hspace{0.1cm} {-} \frac{1}{s_1{+}s_2}\bigg(\frac{1}{g_1}\log\frac{|T{-}s_1|}{|T_0{-}s_1|}
        - \frac{1}{g_2}\log\frac{|T+s_2|}{|T_0+s_2|}\bigg)\\
        &\hspace{0.1cm} 
     + \frac{4s_1s_2}{g_1 g_2\sqrt{g_3}}\left({\arctan}\frac{2T{-}p}{\sqrt{g_3}} {-} {\arctan}\frac{2T_0{-}p}{\sqrt{g_3}}\right) \!\Bigg],
       \end{aligned}
    \end{align}
    }
    where $p=s_2-s_1$, $q = s_1^2 - s_1 s_2 + s_2^2$, $g_1=3s_1^2 - 2 s_1 s_2 + s_2^2$, $g_2 = s_1^2 - 2 s_1 s_2 + 3s_2^2$, and $g_3 = 3s_1^2 - 2s_1s_2 + 3s_2^2$.
\end{theorem}
\begin{proof}
    The proof is presented in Appendix \ref{app:proof_tempsol}
\end{proof}

We now show that the temperature dynamics computed in Theorem \ref{thm:tempsol} are monotonic in time. We first define the steady-state temperature.
\begin{definition}[Steady-state]
The \emph{steady-state temperature} is the limiting temperature of the conductor as $\tau \to \infty$.
\end{definition}
We fix a line \( ij \), and omit the subscript in what follows.

\begin{proposition}\label{prop:time-monotonicity}
The function $T(\tau)$ converges monotonically to $s_1$: $T(\tau) \underset{\tau \to \infty}{\searrow} s_1$ if $T_0 > s_1$, and $T(\tau) \underset{\tau \to \infty}{\nearrow} s_1$ if $T_0 < s_1$.
\end{proposition}
\begin{proof}
From the inverse ODE $\frac{d\tau}{dT} = -\frac{1}{K_4 P(T)}$% with $f(T) = (T - s_1)(T + s_2)(T^2 - pT + q)$
, we see $\frac{d\tau}{dT} < 0$ for $0 \le T < s_1$ and $\frac{d\tau}{dT} > 0$ for $T > s_1$. Since $T > 0$, $\tau(T)$ is increasing for $T < s_1$, decreasing for $T > s_1$, so its inverse $T(\tau)$ increases to $s_1$ when $T_0 <s_1$, and decreases to $s_1$ when $T_0 > s_1$.
\end{proof}
Proposition \ref{prop:time-monotonicity} shows that the steady-state temperature is computed as the positive root \( s_{1,ij} \) of \( P_{ij} \), and can therefore be obtained via a root-finding algorithm, without optimization or ODE-solving software. 
% It is also independent of the initial temperature $T_0$.
% Moreover, the monotonicity of Proposition \ref{prop:time-monotonicity} implies that the maximal temperature of a conductor is attained either at the beginning or the end of the interval, allowing us to consider only these time points. 
% -------------------------------------------------------------------------------------
\subsection{Multi-period Dynamics and Finite Reformulation}\label{sec:ieee_multi_period}
In the multi-period ACOPF framework, the grid operator makes decisions in sequence every $\Delta$ seconds. 
We therefore consider dispatch intervals $((t -1) \Delta, t\Delta]$, indexed by $t \in \sT$. The temperature $T_{ij, t}$ denotes the temperature at the end of the interval $t$, i.e. at time $\tau = t\Delta$.
We denote by $W_{ij, t} = (v_{w, ij, t}, \phi_{ij,t}, T_{a , ij, t})$ the vector of weather conditions on line $ij$ on dispatch interval $t$. Let $f_{W, \Delta} : \RR^+ \times \RR^+ \to \RR^+$ be the solution of eq.~\eqref{eq:temp-dyn-linear-qc}, on an interval of length $\Delta$, with weather $W$. The function $f_{W,\Delta}$ takes as input the temperature at the beginning of the interval and the square magnitude of current on the interval, and maps these to the temperature at the end of the interval. Thus, $f_{W,\Delta}$ provides the temperature update associated with a single dispatch step of duration $\Delta$.
With this notation, for any line $ij$, we get: 
\begin{align}\label{eq:multi-temp_evol}
    T_{ij, t} = f_{W_{ij,t}, \Delta}(\iota_{ij, t}, T_{ij, t{-}1}), \quad \forall t \in \mathcal T.
\end{align}

Using the closed-form solution of Theorem \ref{thm:tempsol}, we reformulate the multi-period evolution \eqref{eq:multi-temp_evol} as:
\begingroup
\allowdisplaybreaks
\begin{subequations}\label{eq:reform}
% TODO: Relabel these
\begin{IEEEeqnarray}{ll}
K_{0,ij,t}= \frac{
  \pi D_0\epsilon\sigma T_{a,ij,t}^4
  {+} r_{ij}\iota_{ij,t} 
  {+} q_{s,ij,t}
  {-} k^{c,0}_{ij, t}}
 {m_{ij}c_{p, ij}}, \; \label{eq:reform:K0}
  &\forall t \in \sT,\IEEEeqnarraynumspace\\
K_{1,ij,t}= \frac{k^c_{ij, t}}{m_{ij}c_{p, ij}},\label{eq:reform:K1}
  &\forall t \in \sT,\IEEEeqnarraynumspace\\
K_{4,ij}=  \frac{\pi D_0\epsilon\sigma}{m_{ij}c_{p, ij}},\label{eq:reform:K4}&\\
g_{1,ij,t}= 3s_{1,ij,t}^2 - 2s_{1,ij,t}s_{2,ij,t} + s_{2,ij,t}^2, \label{eq:reform:g1}
  &\forall t \in \sT,\IEEEeqnarraynumspace\\
g_{2,ij,t}= s_{1,ij,t}^2 - 2s_{1,ij,t}s_{2,ij,t} + 3s_{2,ij,t}^2, \label{eq:reform:g2}
  &\forall t \in \sT,\IEEEeqnarraynumspace\\
g_{3,ij,t}= 3s_{1,ij,t}^2 - 2s_{1,ij,t}s_{2,ij,t} + 3s_{2,ij,t}^2, \label{eq:reform:g3}
  &\forall t \in \sT,\IEEEeqnarraynumspace\\
p_{ij,t}= s_{2,ij,t} - s_{1,ij,t}, \label{eq:reform:p}
  &\forall t \in \sT,\IEEEeqnarraynumspace\\
q_{ij,t}= s_{1,ij,t}^2 - s_{1,ij,t}s_{2,ij,t} + s_{2,ij,t}^2, \label{eq:reform:q}
  &\forall t \in \sT,\IEEEeqnarraynumspace\\
K_{1,ij,t}= K_{4,ij}p_{ij,t}\!\left(s_{1,ij,t}^2 + s_{2,ij,t}^2\right), \label{eq:reform:root1}
  &\forall t \in \sT,\IEEEeqnarraynumspace\\
K_{0,ij,t}= K_{4,ij}q_{ij,t}s_{1,ij,t}s_{2,ij,t}, \label{eq:reform:root2}
  &\forall t \in \sT,\IEEEeqnarraynumspace\\
(T_{ij,t-1}-s_{1,ij,t})e^{h^{1}_{ij,t}}
  = T_{ij,t}-s_{1,ij,t}, \label{eq:reform:r1}
  &\forall t \in \sT,\IEEEeqnarraynumspace\\
(T_{ij,t-1}+s_{2,ij,t})e^{h^{2}_{ij,t}}
  = T_{ij,t}+s_{2,ij,t}, \label{eq:reform:r2}
  &\forall t \in \sT,\IEEEeqnarraynumspace\\
e^{h^{\text{quad}}_{ij,t}} 
  =\frac{T_{ij,t}^2 - p_{ij,t}T_{ij,t} + q_{ij,t}}
        {T_{ij,t-1}^2 - p_{ij,t}T_{ij,t-1}   + q_{ij,t}}, \label{eq:reform:quad}
  &\forall t\in \sT,\IEEEeqnarraynumspace\\[1pt]
K_{4,ij}{=} {}\mathrlap{\begin{aligned}[t]
 &{-}\frac{1}{s_{1,ij,t}+s_{2,ij,t}}
   \Bigl(\frac{h^{1}_{ij,t}}{g_{1,ij,t}}
        -\frac{h^{2}_{ij,t}}{g_{2,ij,t}}\Bigr)\\
 &\quad{+}\frac{s_{2,ij,t}-s_{1,ij,t}}{g_{1,ij,t}g_{2,ij,t}}
   h^{\text{quad}}_{ij,t}\\[1pt]
&\quad{+}\frac{4s_{1,ij,t}s_{2,ij,t}}{g_{1,ij,t}g_{2,ij,t}\sqrt{g_{3,ij,t}}}\Bigl[
  \xi_{ij,t}'{-}\xi_{ij,t}
\Bigr],
   \end{aligned}}& \forall t \in \sT, \label{eq:reform:Tempsol}
\end{IEEEeqnarray}  
\end{subequations}
\endgroup
where we used the shorthand notation $\xi_{ij, t}$ to denote $\arctan\!\bigl(\tfrac{2T_{ij,t-1}-p_{ij,t}}{\sqrt{g_{3,ij,t}}}\bigr)$ and $\xi'_{ij, t}$ to denote $\arctan\!\bigl(\tfrac{2T_{ij,t}-p_{ij,t}}{\sqrt{g_{3,ij,t}}}\bigr)$. We now prove that this model computes the line dynamics.  

\begin{proposition}
    Eq.~\eqref{eq:multi-temp_evol} is equivalent to \eqref{eq:reform}.
\end{proposition}
\begin{proof}
    Equations \eqref{eq:reform:r1}--\eqref{eq:reform:Tempsol} are equivalent to \eqref{eq:solT}, as $\log$ is bijective. Similarly, \eqref{eq:reform:g1}--\eqref{eq:reform:q} define $g_1, g_2, g_3, p, q$ as in Theorem~\ref{thm:tempsol}.
    It remains to show that \eqref{eq:reform:root1}--\eqref{eq:reform:root2} ensure the factorization of the quartic polynomial $f(T) = T^4 + \frac{K_1}{K_4} T - \frac{K_0}{K_4}$
    as $f(T) = (T - s_1)(T + s_2)(T^2 - pT + q)$,
    where $p = s_2 - s_1$ and $q = s_1 s_2 + p(s_2 - s_1)$ by definition. Expanding the right-hand side and matching coefficients yields $\frac{K_1}{K_4} = s_1 s_2 p + q (s_2 - s_1)$ and $\frac{K_0}{K_4} = s_1 s_2 q$
    which are precisely \eqref{eq:reform:root1} and \eqref{eq:reform:root2}. This completes the proof.
\end{proof}
The reformulation \eqref{eq:reform} is composed of a finite number of nonlinear inequalities, instead of ODEs, allowing for explicit implementation using optimization modeling software.

\subsection{Transient and steady-state temperature models}

In the DLR framework, transmission limits are enforced by ensuring that line temperatures remain below the conductor's maximum admissible temperature at all times. The monotonicity of temperature (Proposition~\ref{prop:time-monotonicity}) allows us to only impose:
\begin{align}
    T_{ij, t} \leq T^{\max},&& \forall t \in \sT. \label{eq:temp:temp_limit}
\end{align}
This constraint ensures safe operation by limiting conductor sag and preserving the line's electrical characteristics near nominal conditions \cite{murphy2020temperature}. 
Previous DLR studies relied on steady-state ratings, where the temperature at the end of a dispatch interval is computed as the steady-state temperature. As noted in Proposition \ref{prop:time-monotonicity}, the steady-state temperature can be computed by replacing $\eqref{eq:reform:r1}-\eqref{eq:reform:Tempsol}$ by: 
\begin{align}\label{eq:reform:ss}
        T_{ij,t}&=s_{1,ij,t}, \qquad \forall t \in \mathcal T.
\end{align}

We extend beyond this simplification and model both the steady-state and transient temperature.
To achieve this, we leverage the reformulation Eq.~\eqref{eq:reform}. Since all variables in this reformulation are uniquely determined by the current vector $\iota_{ij, :}$, we define the temperature models using only the variables $\iota_{ij, :}$ and $T_{ij, :}$. 
\begin{enumerate}
    \item \hspace{-0.5em}The \textbf{transient temperature model}, $\mathrm{Tmp}^{\mathrm{Trans}}_{ij}$, is the set:% defined by:
    \begin{equation} \label{eq:trans_model}
    \mathrm{Tmp}^{\mathrm{Trans}}_{ij} = \left\{ (\iota_{ij, :}, T_{ij, :}) \given  \text{Eqs. \eqref{eq:reform} and \eqref{eq:temp:temp_limit}} \right\}.
    \end{equation}

    \item \hspace{-0.8em} The \textbf{steady-state temperature model}, $\mathrm{Tmp}^{\mathrm{SS}}_{ij}$, is the set:
\begin{align} \label{eq:ss_model}
  \hspace{-2em} \mathrm{Tmp}^{\mathrm{SS}}_{ij} =
  \left\{ (\iota_{ij, :}, T_{ij, :}) \given 
  \text{Eqs.~\eqref{eq:reform:K0}-\eqref{eq:reform:root2}, \eqref{eq:temp:temp_limit}, and   \eqref{eq:reform:ss}}
  \right\}.
\end{align}
\end{enumerate}
We use $\mathrm{Tmp}_{ij}$ to denote either of these temperature models. 

The transient and steady-state thermal behavior is illustrated over three time periods in Fig.~\ref{fig:ss-trans-temp}.  

\begin{figure}[h!]
    \vspace{-1em}
    \centering
    \includegraphics[width=0.7\linewidth]{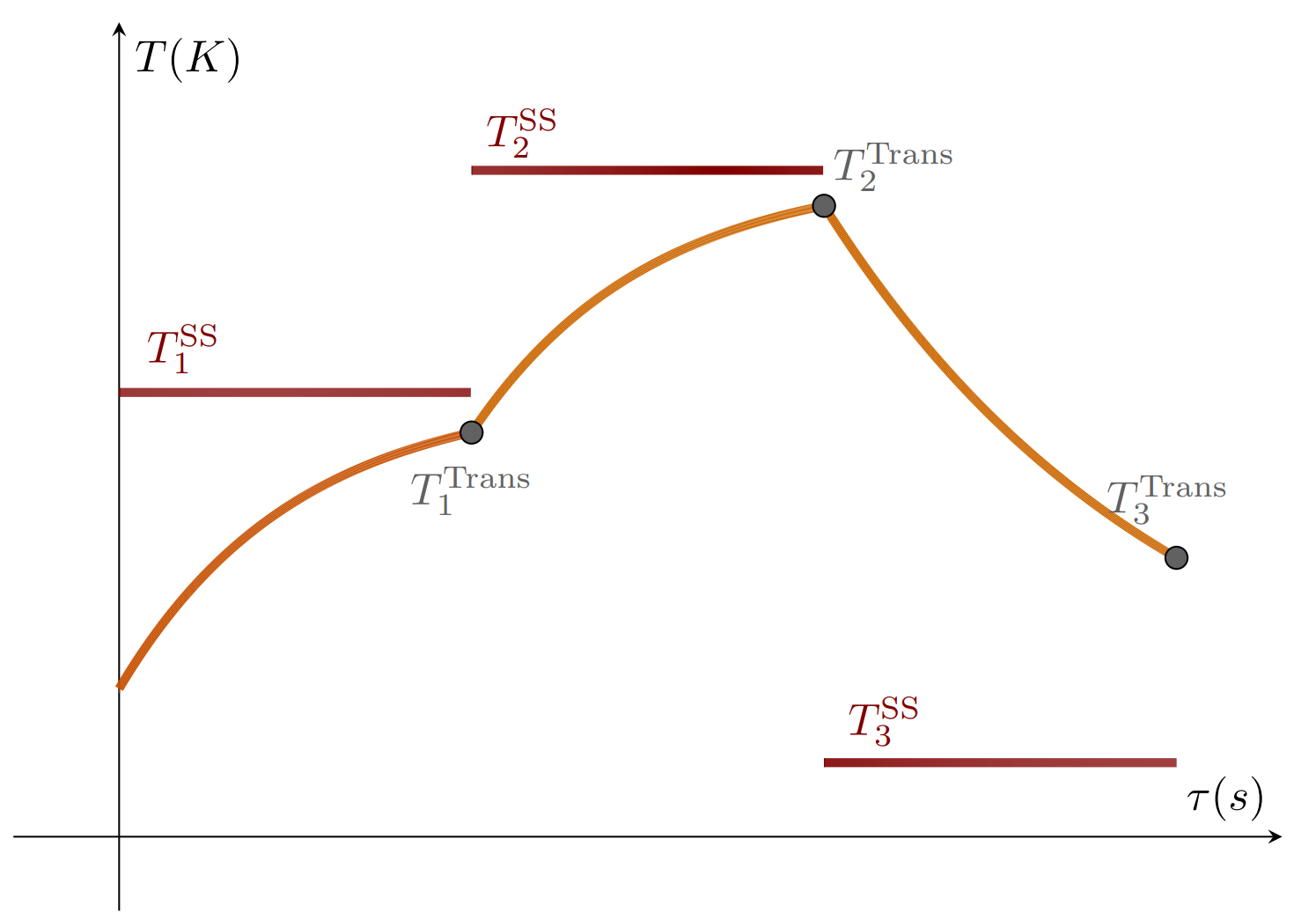}
    \vspace{-1em}
    \caption{Steady-state and transient temperature evolution on three time intervals.}
    \label{fig:ss-trans-temp}
\end{figure}

\section{Multi-period ACOPF with DLR and Ramping constraints}\label{sec:problem}
\noindent
We now present the multi-period ACOPF with DLR, also denoted as DLR-ACOPF. Following the classical multi-period ACOPF framework, we aim to minimize the total generation cost while maintaining admissible currents on the transmission system and satisfying ramping constraints. For any variable $z$, we use \( z_{:, j} \) and \( z_{i, :} \) to denote column and row slices, respectively. DLR-ACOPF writes:
\begin{subequations}\label{eq:synthetic_problem}
\begin{IEEEeqnarray}{rrl}
    \smashoperator[r]{\min_{\substack{%
   (p^{G},q^{G}) \in \mathbb{R}^{|\sG|\times|\sT|}\\
   (e,f)         \in \mathbb{R}^{|\sN|\times|\sT|}\\
   (T,\iota,I^{\mathrm{re}},I^{\mathrm{im}})%
                 \in \mathbb{R}^{|\sB|\times|\sT|}}}}
& \sum_{t\in\sT}\sum_{g\in\sG} C_g\bigl(p^{G}_{g,t}\bigr)
\IEEEeqnarraynumspace \IEEEnonumber\\[4pt]
\text{s.t.}\;
&(p^{G}_{:,t}{,}\, q^{G}_{:,t}{,}\, e_{:,t}{,}\, f_{:,t},%\,
  I^{\mathrm{re}}_{:,t}{,}\, I^{\mathrm{im}}_{:,t}{,}\, \iota_{:,t})
  \in \mathrm{AC}_t, &\, \forall t\in\sT,
\IEEEeqnarraynumspace \IEEEyesnumber \label{eq:synthetic_problem:temp}\\
&(T_{ij,:},\iota_{ij,:}) \in \mathrm{Tmp}_{ij},\;
  & \forall ij\in\sE,
\IEEEeqnarraynumspace \IEEEyesnumber\\
&p^{G}_{g,:} \in \mathrm{Rmp}_g,\; & \forall g\in\sG.%
\IEEEeqnarraynumspace \IEEEyesnumber
\end{IEEEeqnarray}
\end{subequations}

The set $\mathrm{AC}_t$ denotes the following set of AC-OPF constraints:
\begin{subequations}\label{eq:ACOPF_rect}
\allowdisplaybreaks
\begin{IEEEeqnarray}{ll}
p_{i,t} = \bar G_i\nu_{i,t}
  +\sum_{j\in\sN_i}(C_{ij,t}G_{ij}+S_{ij,t}B_{ij}),\quad\label{eq:ac:real_power_balance}
  &\forall i\in\sI,\IEEEeqnarraynumspace\\[1pt]
q_{i,t}{=}{-}\bar B_i\,\nu_{i,t}
  +\sum_{j\in\sN_i}(S_{ij,t}G_{ij}{-}C_{ij,t}B_{ij}),\label{eq:ac:im_power_balance}
  &\forall i\in\sI,\IEEEeqnarraynumspace\\[1pt]
p_{i,t}= \sum_{g\in G_i}p^G_{g,t}-P^D_{i,t},\label{eq:AC:real_net_generation}
  &\forall i\in\sI,\IEEEeqnarraynumspace\\
q_{i,t}= \sum_{g\in G_i}q^G_{g,t}-Q^D_{i,t},\label{eq:AC:imaginary_net_generation}
  &\forall i\in\sI,\IEEEeqnarraynumspace\\
(V_i^{\min})^2\le \nu_{i,t}\le (V_i^{\max})^2,\label{eq:ac:voltage_limits}
  &\forall i\in\sI,\IEEEeqnarraynumspace\\
\theta_{ij}^{\min}\le \phi_{i,t}-\phi_{j,t}\le\theta_{ij}^{\max},\label{eq:ac:angle_limits}
  &\forall ij\in\sB,\IEEEeqnarraynumspace\\
P_g^{\min}\le p^G_{g,t}\le P_g^{\max},\label{eq:ac:real_limits}
  &\forall g\in\sG,\IEEEeqnarraynumspace\\
Q_g^{\min}\le q^G_{g,t}\le Q_g^{\max},\label{eq:ac:reactive_limits}
  &\forall g\in\sG,\IEEEeqnarraynumspace\\
I_{ij,t}^{\mathrm{re}}=G_{ij}(e_{i,t}-e_{j,t})-B_{ij}(f_{i,t}-f_{j,t}),
  \label{eq:ac:real_current}&\forall ij\in\sE,\IEEEeqnarraynumspace\\
I_{ij,t}^{\mathrm{im}}=G_{ij}(f_{i,t}-f_{j,t})+B_{ij}(e_{i,t}-e_{j,t}), \quad
  \label{eq:ac:imaginary_current}&\forall ij\in\sE,\IEEEeqnarraynumspace\\
\iota_{ij,t}= (I_{ij,t}^{\mathrm{re}})^{2}+(I_{ij,t}^{\mathrm{im}})^{2},
  \label{eq:ac:current_magnitude}&\forall ij\in\sE, \IEEEeqnarraynumspace
\end{IEEEeqnarray}
\end{subequations}
where for notational convenience we used $\nu_{i,t}=e_{i,t}^{2}+f_{i,t}^{2}$,
$C_{ij,t}=e_{i,t}e_{j,t}+f_{i,t}f_{j,t}$,
$S_{ij,t}=f_{i,t}e_{j,t}-e_{i,t}f_{j,t}$, $\bar G_i = G_i + G_i^{s}$, $\bar B_i = B_i + B_i^{s}$, and $\phi_{i,t} = \operatorname{atan2}(f_{i,t},e_{i,t})$.

In this formulation, eqs.~\eqref{eq:ac:real_power_balance} and~\eqref{eq:ac:im_power_balance} are the power balance equations. Eqs.~\eqref{eq:AC:real_net_generation} and~\eqref{eq:AC:imaginary_net_generation} compute the demand at each bus. Eqs.~\eqref{eq:ac:voltage_limits}-\eqref{eq:ac:angle_limits} are the voltage and angle limits on the lines, respectively. Eqs.~\eqref{eq:ac:real_limits} and~\eqref{eq:ac:reactive_limits} are the generation limits of the generators. Finally, eqs.~\eqref{eq:ac:real_current} and~\eqref{eq:ac:imaginary_current} compute the real and imaginary part of the current, and eq.~\eqref{eq:ac:current_magnitude} computes the current magnitude squared. 

\noindent
The ramping constraints of generator $g \in \sG$, $\mathrm{Rmp}_g$ write:
\begin{align}
     {P^{Rmp, -}_g} \leq p^G_{g, t} - p^G_{g, t-1} \leq {P^{Rmp, +}_g}, \quad& t \in \sT'.
    \label{eq:ac:Ramping}
\end{align}

We now prove that problem \eqref{eq:synthetic_problem} with $\mathrm{Tmp}^{\mathrm{Trans}}{ij}$ is a relaxation of problem \eqref{eq:synthetic_problem} with $\mathrm{Tmp}^{\mathrm{SS}}{ij}$.

\begin{proposition} \label{prop:ss_restriction}
    Let $v^{\mathrm{Trans}}, v^{\mathrm{SS}}$ be the objective value to Problem \eqref{eq:synthetic_problem} with \eqref{eq:synthetic_problem:temp} as $\mathrm{Tmp}^{\mathrm{Trans}}_{ij}$ and $\mathrm{Tmp}^{\mathrm{ss}}_{ij}$ respectively. Then, $v^{\mathrm{Trans}} \leq v^{\mathrm{SS}}$.
\end{proposition}

\begin{proof}
We prove that any feasible solution to the steady-state model yields a feasible solution to the transient model with the same objective value.

Let $x^{\mathrm{SS}} = (p^{G, SS}, q^{G, SS}, e^{\mathrm{SS}}, f^{\mathrm{SS}}, \iota^{\mathrm{SS}}, I^{\mathrm{re}}, I^{\mathrm{im}}, T^{\mathrm{SS}})$ be a feasible steady-state solution. Define  
$x^{\mathrm{Trans}} := (p^{G, SS}, q^{G, SS}, e^{\mathrm{SS}}, f^{\mathrm{SS}}, \iota^{\mathrm{SS}}, I^{\mathrm{re}}, I^{\mathrm{im}}, T'),$
where $T'$ is computed via~\eqref{eq:reform:K0}–\eqref{eq:reform:Tempsol}, using the same current values $\iota_{ij,t}^{\mathrm{SS}}$.

We show by induction on $t \in \{0\} \cup \sT$ that $T'_{ij, t} \leq T^{\max}$ for all lines $(i,j)$.
First, for $t=0$, $T'_{ij,0} = T_{ij,0}^{\mathrm{SS}} = T_0 \leq T^{\max}$.

Now, fix $t\geq 1$ and assume $T'_{ij,t-1} \leq T^{\max}$. Consider two cases:
\begin{itemize}
    \item If $T'_{ij,t-1} \leq T^{\mathrm{SS}}_{ij,t}$ (heating), then by Proposition \ref{prop:time-monotonicity}, $T'_{ij,t} \leq T^{\mathrm{SS}}_{ij,t} \leq T^{\max}$.
    \item If $T'_{ij,t-1} \geq T^{\mathrm{SS}}_{ij,t}$ (cooling), then again by Proposition \ref{prop:time-monotonicity} and the induction hypothesis, $T'_{ij,t} \leq T'_{ij,t-1} \leq T^{\max}.$
\end{itemize}

In both cases, $T'_{ij,t} \leq T^{\max}$. By induction, this holds for all $t$.
Hence, $x^{\mathrm{Trans}}$ is feasible in the transient model. Since the objective values match, the result follows.
\end{proof}
This result shows that the steady-state approximation used in previous DLR studies is a safe approximation of the line dynamics. As noted in Proposition \ref{prop:time-monotonicity}, the steady-state temperature computations decouple the time-steps: $(\iota, T) \mapsto f_{W_{ij,t{-}1}, \infty}(\iota, T)$ is a continuous increasing function in $\iota$, and is constant in $T$. We can therefore define the maximum steady-state current on a line $ij$ and time-period $t$ as:
\begin{align} \label{eq:Imax}
    \iota^{\max}_{ij, t}:= \frac{K_{4 \, ij, t} (T^{\max})^4 + K_{1 \, ij, t} T^{\max} - K_{0 \, ij, t}'}{r_{ij}},
\end{align}

and $(T_{ij, :}, \iota_{ij, :}) \in \mathrm{Tmp}^{\mathrm{ss}}_{ij}$ is equivalent to:
\begin{align}\label{eq:model:current_limit}
    \iota_{ij, t} \leq \iota^{\max}_{ij, t}, \quad& \forall t \in \sT.
\end{align}
As a result, we can equivalently impose current limits directly in the multi-period ACOPF model to enforce steady-state DLR.

\section{Decomposition Strategy and Bi-level ADMM}\label{sec:decomposition}
\noindent

In this section, we present our space-time decomposition and the bi-level ADMM algorithm used to solve our problem.
\subsection{Decomposition strategy}\label{subsec:decompositon}

Problem~\eqref{eq:synthetic_problem} involves several ACOPF and temperature dynamics equations, non-convex constraints that couple variables across both time and space. In each time period, the transmission constraints~\eqref{eq:ac:real_power_balance}–\eqref{eq:ac:im_power_balance} link generation, current, and voltage variables. Across time, each line is coupled by the transient temperature block $\mathrm{Tmp}^{\mathrm{Trans}}_{ij}$, while each generator is subject to the convex ramping constraint~\eqref{eq:ac:Ramping}. In contrast, the constraint sets $\mathrm{AC}_t$ and $\mathrm{AC}_{t'}$ are independent for distinct time steps $t \neq t'$. Likewise, the temperature and ramping constraints are separable across lines and generators, respectively.
As solving \eqref{eq:synthetic_problem} is intractable even for small networks, we propose a decomposition method that leverages parallelization within an ADMM framework.%($12$ time periods, $60$ buses)

Starting from formulation~\eqref{eq:synthetic_problem}, we introduce variables $\iota^{\mathrm{Tmp}}$, $\iota^{\mathrm{AC}}$, $p^{G, \mathrm{Rmp}}$, and $p^{G, \mathrm{AC}}$, which serve as local copies of the variables involved in the  $\mathrm{Tmp}$, $\mathrm{AC}$, and $\mathrm{Rmp}$ blocks, respectively. 
To enforce consistency across these blocks, we add the following consensus constraints:
\begin{equation}
    \iota^{\mathrm{AC}} = \iota^{\mathrm{Tmp}}, \quad 
    p^{G, \mathrm{AC}} = p^{G, \mathrm{Rmp}}. \label{eq:consensus}
\end{equation}
% However, as noted in~\cite{xu2020us}, applying ADMM to nonconvex problems does not guarantee convergence when the final block of variables is constrained. 
%Therefore, we introduce slack variables $\alpha$ and $\beta$, and equivalently write~\eqref{eq:synthetic_problem} as:
This yields the equivalent reformulation of \eqref{eq:synthetic_problem}:
\begin{subequations} \label{eq:model_with_slacks}
	\allowdisplaybreaks
	\begin{IEEEeqnarray}{lrl}
	% ------------------------ objective ---------------------------------
	\smashoperator[r]{\min_{\substack{%
	(p^{G,\mathrm{Rmp}},q^{G},p^{G,\mathrm{AC}})\in\mathbb{R}^{|\sG|\times|\sT|}\\
	(e,f)\in\mathbb{R}^{|\sN|\times|\sT|}\\
	(T,I^{\mathrm{re}},I^{\mathrm{im}},\iota^{\mathrm{Tmp}},\iota^{\mathrm{AC}})
			\in\mathbb{R}^{|\sB|\times|\sT|}}}}
	& &\mathclap{\quad 
	\sum_{t\in\sT}\sum_{g\in\sG} C_g\bigl(p^{G}_{g,t}\bigr)} 
	\IEEEeqnarraynumspace \IEEEyesnumber
	\label{eq:naive_ADMM_obj}\\[1ex]
	% ------------------------ constraints -------------------------------
	&\hphantom{\makebox[0.36\linewidth][l]{}}&\IEEEnonumber\\[-1.6ex]
	\text{s.t.}
	& \mathllap{(p^{G,\mathrm{AC}}_{:,t},q^{G}_{:,t},e_{:,t},f_{:,t}, \iota^{\mathrm{AC}}_{:,t}, I^{\mathrm{re}}_{:,t},
	I^{\mathrm{im}}_{:,t})} & \in \mathrm{AC}_{t}, \forall t\in\sT,
	\IEEEeqnarraynumspace \IEEEyesnumber
	\label{eq:ADMM:AC:2}\\
	& (T_{ij,:}, \iota^{\mathrm{Tmp}}_{ij,:})& 
	\in \mathrm{Tmp}_{ij}, \forall ij{\in}\sE,
	\IEEEeqnarraynumspace \IEEEyesnumber
	\label{eq:ADMM_inner:Tmp:2}\\
	& p^{G,\mathrm{Rmp}}_{g,:}& 
	\in \mathrm{Rmp}_{g}, \forall g\in\sG,
	\IEEEeqnarraynumspace \IEEEyesnumber
	\label{eq:ADMM_inner:Ramping:2}\\
    &\eqref{eq:consensus}. \IEEEnonumber
	%
	% & I^{\mathrm{Tmp}}
	% &= I^{\mathrm{AC}} %+ \alpha, 
	% \IEEEeqnarraynumspace \IEEEyesnumber \label{eq:ADMM_inner:slack:1}\\
	% %
	% & p^{G,\mathrm{Rmp}}
	% &= p^{G,\mathrm{AC}} %+ \beta, 
	% \IEEEeqnarraynumspace \IEEEyesnumber
	% \label{eq:ADMM_inner:slack:2}\\
	%
	% & \alpha & = 0, 
	% \IEEEeqnarraynumspace \IEEEyesnumber
	% \label{eq:ADMM_outer_slack:1}\\
	% %
	% & \beta  &= 0.
	% \IEEEeqnarraynumspace \IEEEyesnumber
	% \label{eq:ADMM_outer_slack:2}
	\end{IEEEeqnarray}
\end{subequations}
% We now relax \eqref{eq:ADMM_outer_slack:1} and \eqref{eq:ADMM_outer_slack:2} in the Augmented Lagrangian framework, with dual variables $w$, and penalty parameter $\theta$ to get:

% \begin{subequations} \label{eq:alm_relaxation_1}
% 	\allowdisplaybreaks
% 	\begin{IEEEeqnarray}{lrl}
% 	% ------------------------ objective ---------------------------------
% 	\smashoperator[r]{\min_{\substack{%
% 	(p^{G,\mathrm{Rmp}},q^{G},p^{G,\mathrm{AC}}, \beta)\in\mathbb{R}^{|\sG|\times|\sT|}\\
% 	(e,f)\in\mathbb{R}^{|\sN|\times|\sT|}\\
% 	(T,\iota^{\mathrm{re}},\iota^{\mathrm{im}},I^{\mathrm{Tmp}},I^{\mathrm{AC}}, \alpha)
% 			\in\mathbb{R}^{|\sB|\times|\sT|}}}}
% 	& &\mathclap{\quad 
% 	\sum_{t\in\sT}\sum_{g\in\sG} C\bigl(p^{G}_{g,t}\bigr) + w_{\alpha}^\top \alpha + w_{\beta}^\top \beta + \frac{\theta}{2}\norm{\begin{pmatrix}
% 	    \alpha \\
%         \beta
% 	\end{pmatrix}}}
% 	\IEEEeqnarraynumspace \\[1ex]
% 	% ------------------------ constraints -------------------------------
% 	&\hphantom{\makebox[0.345\linewidth][l]{}}&\IEEEnonumber\\[-1.6ex]
% 	\text{s.t.}
% 	\eqref{eq:ADMM:AC:2}, \eqref{eq:ADMM_inner:Tmp:2},	\eqref{eq:ADMM_inner:Ramping:2}, \eqref{eq:ADMM_inner:slack:1}, \eqref{eq:ADMM_inner:slack:2}
%     \end{IEEEeqnarray}
% \end{subequations}
% The Augmented Lagrangian of \eqref{eq:model_with_slacks} now writes:
% \[
%     \mathcal{L}(\dots) = C\bigl(p^{G}_{g,t}\bigr
% \]
We now introduce compact notations for the model variables. For any matrix \( z \in \R^{m \times n} \), let \( \tilde z \in \R^{mn} \) denote the column-wise vectorization, i.e., \(z_{i, j} = \tilde z_{n(j-1) + i}.\)

Define the temporally and spatially coupled variable vectors:
\[
\begin{aligned}
    x &:= (\tilde p^{G, AC}, \tilde \iota^{AC}, \tilde q^{G}, \tilde e, \tilde f, \tilde I^{\mathrm{Re}}, \tilde I^{\mathrm{Im}}) \in \R^{|\sT| \cdot (3|\sB| + 2|\sG| + 2|\sN|)}, \\
    y &:= (\tilde p^{G, \mathrm{Rmp}}, \tilde \iota^{\mathrm{Tmp}}, \tilde T^{\mathrm{Tmp}}) \in \R^{|\sT| \cdot (2|\sB| + |\sG|)}.
\end{aligned}
\]
Let \( x_t := (p^{G, AC}_{:, t}, \iota^{AC}_{:, t}, q^{G}_{:, t}, e_{:, t}, f_{:, t}, I^{\mathrm{Re}}_{:, t}, I^{\mathrm{Im}}_{:, t}) \) denote the temporal slice of \( x \), and define the spatial slice of \( y \) as:
\[
y_l := 
\begin{cases}
    (\iota^{\mathrm{Tmp}}_{l, :},\ T_{l, :}) & \text{if } l \in \sE, \\
    (p^{G,\mathrm{Rmp}}_{l, :}) & \text{if } l \in \sG.
\end{cases}
\]

Let \( \mathrm{Dev}_l \) be the set of device-level constraints: \( \mathrm{Tmp}_l \) if \( l \in \sE \), and \( \mathrm{Rmp}_l \) if \( l \in \sG \).
We now introduce component selection matrices $A, B$ as:
\begin{align*}
    A &= \begin{pmatrix}
\begin{matrix} -\mathbb I_{|\sT| (|\sG| + |\sE|)} \end{matrix} & \mathbf 0
\end{pmatrix} \in \R^{|\sT|  (|\sG| +|\sE|) {\times}|\sT|  (2 |\sG| + 3  |\sB| + 2 |\sN|) }, \\
    B &= \begin{pmatrix}
    \begin{matrix} \mathbb I_{|\sT|(|\sG| + |\sE|)} \end{matrix} & \mathbf 0
    \end{pmatrix} \in \R^{(|\sG| + |\sE| ) |\sT| \times(|\sG| + 2 |\sB| ) |\sT|} .
\end{align*}

Problem~\eqref{eq:model_with_slacks} can be equivalently reformulated as:

\begin{subequations}\label{eq:compact_original_separable_problem}
\begin{align}
     \underset{
            \substack{x, y}
        }{\text{minimize}}\quad & C(x)\\
        \text{s.t.} \quad  & x_t \in \mathrm{AC}_t,& \forall t \in \sT \\ 
        & y_l \in\mathrm{Dev}_l, & \forall l \in \sE \cup \sG \\
        &  Ax + By=0. \label{eq:compact_original_problem_consensus}
		% & u = 0. \label{eq:compact_original_problem_slack}
\end{align}
\end{subequations}

 However, as noted in~\cite{xu2020us}, applying ADMM to nonconvex problems does not guarantee convergence when the final block of variables is constrained. We will instead apply ADMM to an augmented Lagrangian relaxation of \eqref{eq:model_with_slacks}.
 We introduce the slack variable vector $u$, and rewrite \eqref{eq:compact_original_problem_consensus} as:
 \begin{align}
     Ax + By + u = 0 \label{eq:consensus_slack}\\
     u=0 \label{eq:zeros_slacks}
 \end{align}

We now relax \eqref{eq:zeros_slacks} in the augmented Lagrangian framework. Introduce dual variable $w$ and penalty parameter $\theta$, to obtain the following augmented Lagrangian relaxation of~\eqref{eq:compact_original_separable_problem}:
\begin{subequations}\label{eq:compact_augmented_lagrangian}
	\begin{align}
    \underset{
            \substack{x, y, u}
        }{\text{minimize}}\quad & C(x) + {w}^\top u + \frac{\theta}{2} \norm{u}_2^2 \\
        \text{s.t.} \quad  & x_t \in \mathrm{AC}_t,& \forall t \in \sT, \\ 
        & y_l \in \mathrm{Dev}_{l}, & \forall l \in \sE \cup \sG, \\
        &\eqref{eq:consensus_slack}. \notag 
\end{align}
\end{subequations}
Problem~\eqref{eq:compact_augmented_lagrangian} is separable into three blocks of variables $x, y$ and $u$, where the third block in $u$ is unconstrained with a smooth objective $\theta/2\|u\|^2_2$, facilitating the convergence of the ADMM algorithm. We define $p := Ax + By + u$, and we introduce the multiplier $v$ and penalty $\rho$ to relax~\eqref{eq:compact_original_problem_consensus} in the augmented Lagrangian framework, forming the objective: 
\[
   \mathcal{L}(x, y, u, v, w) =  C(x) + {w}^\top u + \frac{\theta}{2} \norm{u}_2^2 + {v}^\top p +\frac{\rho}{2} \norm{p}_2^2. 
\]
% \begin{definition}($\epsilon$-stationarity for~\eqref{eq:compact_original_separable_problem})
% Let $L(x, y, v) := C(x) + \delta_{\mathrm{AC}}(x) + \delta_{\mathrm{Temp} \times \mathrm{Rmp}}(y) + v^\top(Ax + By)$.
%     We say that the primal dual point $(x, y, z)$ is an $\epsilon$-stationary solution of~\eqref{eq:compact_original_separable_problem} if there exist $(d_1, d_2, d_3)$ with $\max(\norm{d_1}, \norm{d_2}, \norm{d_3}) \leq \epsilon$ such that $(d_1, d_2, d_3) \in \partial L(x, y, z)$, or equivalently:
%     \begin{align}
%         &d_1 = Ax + By \label{eq:primal_feasiblity},\\
%         &d_2 \in B^\top v + \sN_{\mathrm{Temp} \times \mathrm{Rmp}}(y)  \label{eq:stationarity_y},\\
%         &d_3 \in A^\top v + \partial(C(x) + \delta_{\mathrm{AC}}(x)) \label{eq:stationarity_x},
%     \end{align}
% where $\sN_{X}(x)$ denotes the generalized normal cone of $X$ at $x$, and $\delta_X(\cdot)$ is the indicator function of the set $X$.
% \end{definition}

% \begin{definition}($\epsilon$-stationarity for~\eqref{eq:compact_augmented_lagrangian})
%     We say that the primal dual point $(x, y, u, w, \theta, v, \rho)$ is an $\epsilon$-stationary solution of~\eqref{eq:compact_original_separable_problem} if there exist $(d_1, d_2, d_3)$ with $\max(\norm{d_1}, \norm{d_2}, \norm{d_3}) \leq \epsilon$ such that~\eqref{eq:stationarity_y},~\eqref{eq:stationarity_x} hold and:
%     \begin{align}
%         &d_1 = Ax + By + u,\\
%         &0 = w + \theta u + v  \label{eq:stationarity_u}.
%     \end{align}
% \end{definition}

\subsection{Bi-level ADMM Algorithm}\label{sec:admm}
In this section, we present a bi-level algorithm used to solve problem~\eqref{eq:compact_original_separable_problem}.
Throughout this section, we adopt the definitions of approximate stationarity as introduced in \cite[Def 1 and 2]{gholami_admm-based_2023}.

We begin with the inner-level ADMM algorithm, which finds an $\epsilon$-stationary point of~\eqref{eq:compact_augmented_lagrangian}, given fixed outer-level variables $w$ and $\theta$.
Let $(r)$ denote the inner iteration index. The variable $x^{k,(r)}$ refers to the state at the $r^{\text{th}}$ inner iteration within the $k^{\text{th}}$ outer iteration. Variables fixed during the inner loop are denoted by $x^k$ at the start of outer iteration $k$, and we omit the superscript $k$ when unambiguous. Upon completion of the inner loop, the resulting solutions are labeled $x^{(k)}, y^{(k)}, u^{(k)}, v^{(k)}$. 
The inner loop solves the augmented Lagrangian relaxation~\eqref{eq:compact_augmented_lagrangian} using a 3-block ADMM scheme, decomposing the problem into subproblems at each inner iteration $r$.

First, optimize over the ACOPF variables independently and in parallel for each time step $t \in \sT$ by solving:
\begin{align*}
x^{(r)}_{:, t} \in \arg\min_{x_{:, t} \in \mathrm{AC}t}
\mathcal{L}(x_{:, t}, y^{(r-1)}, u^{(r-1)}, v^{(r-1)}, w), \forall t \in \sT.
\end{align*}
We only require a stationary point to guarantee convergence. Next, optimize over the temporally coupled variables independently and in parallel for each line and generator $l \in \sE \cup \sG$:
\begin{align*}
y^{(r)}_{l, :} \in \arg\min_{y \in \mathrm{Dev}_{l}}
\mathcal{L}(x^{(r)}, y_{l, :}, u^{(r-1)}, v^{(r-1)}, w), \forall l \in \sE \cup \sG.
\end{align*}
Then, compute slack variables to minimize primal infeasibility:
\begin{align*}
u^{(r)} &= \frac{-v^{(r-1)} - w^{(k)} - \rho (A x^{(r)} + B y^{(r)})}{\rho^{(k)} + \theta^{k}}.
\end{align*}
Finally, perform a dual ascent step $v^{(r)} = v^{(r-1)} + \rho^{(k)} p^{(r)}$.

The complete inner ADMM procedure is summarized in Algorithm \ref{alg:inner-iteration}. The algorithm terminates when the primal residual $\norm{p^{(r)}}_2$ falls below a tolerance threshold.
\begin{algorithm}[t]
\caption{ADMM with 3 blocks for the $k$-th outer level iterate}
\begin{algorithmic}[1]  
% --------------------------------------------------------------------
\REQUIRE Initial primal variables
         $x^{(0)}_{t}$, $y^{(0)}_{ij}$, $y^{(0)}_{g}$, $u^{(0)}$;
         duals $v^{(0)}$, $w^{(k)}$ with
         $w^{(k)} + \theta^{k}u^{(0)} = v^{(0)}$;
         penalty $\rho^{(k)} = 2\theta^{k}$;
         $d = \dim(Ax)$
\ENSURE  Sequences
         $x^{(r)}_{t}$, $y^{(r)}_{ij}$, $y^{(r)}_{g}$, $u^{(r)}$
% --------------------------------------------------------------------
\STATE $r \gets 1$
\WHILE{$\lVert x^{(r)} - y^{(r)} - u^{(r)} \rVert_{2}
       \le \max\!\bigl(\varepsilon,\;\sqrt{d}/(k\sqrt{\rho})\bigr)$}
  % ------------------- x–update -------------------------------------
  \FORALL{$t \in \sT$}
    \STATE
      $x^{(r)}_{t} \in
        \displaystyle\arg\min_{x_{t} \in \mathrm{AC}_{t}}
        \mathcal{L}\bigl(x_{t}, y^{(r-1)}, u^{(r-1)},
                         v^{(r-1)}, w^{(k)}\bigr)$
  \ENDFOR
  % ------------------- y_{ij}–update --------------------------------
  \FORALL{$ij \in \sE$}
    \STATE
      $y^{(r)}_{ij} \in
        \displaystyle\arg\min_{y_{ij} \in \mathrm{Temp}_{ij}}
        \mathcal{L}\bigl(x^{(r)}, y_{ij}, u^{(r-1)},
                         v^{(r-1)}, w^{(k)}\bigr)$
  \ENDFOR
  % ------------------- y_{g}–update ---------------------------------
  \FORALL{$g \in \sG$}
    \STATE
      $y^{(r)}_{g} \in
        \displaystyle\arg\min_{y_{g} \in \mathrm{Rmp}_{g}}
        \mathcal{L}\bigl(x^{(r)}, y_{g}, u^{(r-1)},
                         v^{(r-1)}, w^{(k)}\bigr)$
  \ENDFOR
  % ------------------- u–update -------------------------------------
  \STATE
    $u^{(r)} \in
      \displaystyle\arg\min_{u}
      \mathcal{L}\bigl(x^{(r)}, y^{(r)}, u,
                       v^{(r-1)}, w^{(k)}\bigr)$
  % ------------------- dual update ----------------------------------
  \STATE
    $v^{(r)} \gets
      v^{(r-1)} + \rho^{(k)}
        \bigl(B y^{(r)} + A x^{(r)} + u^{(r)}\bigr)$
  % ------------------- iterate --------------------------------------
  \STATE $r \gets r + 1$
\ENDWHILE
\end{algorithmic}
\label{alg:inner-iteration}
\end{algorithm}

After the inner-level ADMM finds an $\epsilon$-stationary point of~\eqref{eq:compact_augmented_lagrangian}, we exit the inner loop and perform a dual ascent step on the outer level variables $w^k$ as:
\begin{align*}
   \hat w^{k+1}= w^k + \theta^k u^k, \quad  w^{k+1} = Proj_{[\underline w, \overline w]}(\hat w^{k+1}).
\end{align*}
The projection step guarantees convergence of the algorithm (see Proposition \ref{prop:outer_convergence}) by guaranteeing boundedness of the dual multipliers. However, in our experiments, it is not needed. 

Then, increase the penalty parameter $\theta^k$ if sufficient slack reduction $u$ has been made during the previous algorithm iteration. For given scalars $\omega, \gamma$, we update $\theta^k$ as:
\begin{align*}
    \theta^{k+1} = \begin{cases}
        \theta^k & \text{if } \norm{u^k}_2 \leq \omega \norm{u^{k-1}}_2, \\
        \gamma \cdot \theta^{k+1} & \text{otherwise.}
    \end{cases}
\end{align*}

We present the outer iteration algorithm in Algorithm \ref{alg:outer-iteration}.
\begin{algorithm}[t]
\caption{Bi-level ADMM algorithm}

\begin{algorithmic}[1]
  % ----------------------------- I/O ---------------------------------
  \REQUIRE Parameters $\gamma$, $\theta^{0}$, $\omega$; initial primal
           variables $x^{(0)}_{t}\!\!,y^{(0)}_{ij}\!\!,y^{(0)}_{g}\!\!,u^{(0)}$;
           duals $v^{(0)}\!\!, w^{(0)}$ s.t.
           $w^{(0)} {+} \theta^{0}u^{(0)} {=} v^{(0)}$.
  \ENSURE  Stationary point
           $x^{(K)}_{t}\!\!,y^{(K)}_{ij}\!\!,y^{(K)}_{g}\!\!,u^{(K)}$
           of problem \eqref{eq:compact_augmented_lagrangian}
  % ----------------------------- init --------------------------------
  \STATE $k \leftarrow 1$
  % ----------------------------- loop --------------------------------
  \WHILE{$\lVert Ax^{(k)} + By^{(k)} \rVert_{2} \le \varepsilon$}
    % \STATE \COMMENT{inner step — see Algorithm~\ref{alg:inner-iteration}}
    \STATE Run Alg.~\ref{alg:inner-iteration} on
           $(x^{(k-1)}_{t}\!\!, y^{(k-1)}_{ij}\!\!, y^{(k-1)}_{g}\!\!, u^{(k-1)}\!\!, v^{(k-1)}\!\!, w^{(k-1)})$
           s.t.\ $w^{(k-1)} + \theta^{k-1}u^{(k-1)} = v^{(k-1)}$
    \STATE $w^{(k)} \gets w^{(k-1)} + \theta^{k-1}u^{(k)}$
    \IF{$\lVert u^{(k)} \rVert \ge \omega\,\lVert u^{(k-1)} \rVert$}
      \STATE $\theta^{k} \gets \gamma\,\theta^{k-1}$
    \ENDIF
    \STATE $k \gets k + 1$
  \ENDWHILE
  % ----------------------------- return ------------------------------
  \RETURN $(x^{(K)}_{t},\,y^{(K)}_{ij},\,y^{(K)}_{g},\,u^{(K)})$
\end{algorithmic}
\label{alg:outer-iteration}
\end{algorithm}

\vspace{-1em}
\section{Convergence of the Algorithm}\label{sec:convergence}
\noindent

We now prove convergence of the bi‑level ADMM scheme under mild
assumptions, matching the rates of \(O(1/\epsilon^{4})\) and \(O(1/\epsilon^{3})\) proven in \cite{sun_two-level_2023}. 
% Section~\ref{sec:inner_convergence} shows that Algorithm~\ref{alg:inner-iteration} reaches an $\epsilon$‑stationary point of~\eqref{eq:compact_augmented_lagrangian} in \(O(1/\epsilon^{2})\) iterations, exploiting geometric properties of the temperature dynamics.  
% Section~\ref{sec:global} extends the result to the outer loop(Algorithm~\ref{alg:outer-iteration}) for problem~\eqref{eq:compact_original_separable_problem}, with a rate of \(O(1/\epsilon^{4})\) or \(O(1/\epsilon^{3})\).
\subsection{Inner Convergence}
\label{sec:inner_convergence}

Algorithm~\ref{alg:inner-iteration}, adapted from \cite{sun_two-level_2023}, provides an \(O(1/\epsilon^{2})\) rate when the second block is convex.
Although the ACOPF and temperature blocks are nonconvex, the
temperature dynamics are almost linear, deviating only through a
quartic perturbation:
\[
  \frac{dT}{d\tau}
  = \underbrace{K_{0}'-K_{1}T}_{\text{linear}}
  \;-\;
  \underbrace{K_{4}T^{4}}_{\text{quartic}}
  \;+\;
  \underbrace{r'\iota}_{\text{linear input}},
\]
with constants as defined in Theorem \ref{thm:tempsol}. Let $\iota_{<t}$ denote $\iota_1, \ldots \iota_{t-1}$. We define the multi-step solutions \(\{\Phi^{t}\}_{t \in \sT}\) as:
\begin{align*}
\Phi^1(T_0, \iota_1) &= f_{W_1,\Delta}(T_0, \iota_1),\\[2pt]
\Phi^t(T_0, \iota_{<t}) &= f_{W_t,\Delta}(\Phi^{t-1}(T_0, \iota_{<t-1}), \iota_{t-1}),
\text{for } t \in \sT'.
\end{align*}

With this representation, the temperature model \eqref{eq:trans_model} writes:
\begin{align*}
  \Phi^{t}(T_{ij,0},\iota_{ij,1},\ldots,\iota_{ij,t-1})\le T^{\max},
  \qquad \forall t\in \sT.
\end{align*}

In Lemma \ref{thm:Hessian}, we prove that the functions $\Phi^t$ are close to linear in the sense that their curvature is low. 

\begin{lemma}[Uniform second‑order bound]%
\label{thm:Hessian}
Fix $T_0,K_{1}$, $K_{4},r>0$ and assume \(T(t)\le T^{\max}\).
With
\begin{gather*}
\beta:=12K_{4}T^{\max, 2},\;
\underline\kappa:=K_{1}+4K_{4}T^{\max, 3},\;
G_{\Delta}:=e^{-K_{1}\Delta},\\
M_{\Delta}:=\frac{\beta}{\underline\kappa}
            \Bigl(1+\frac{r^{2}}{K_{1}^{2}}\Bigr)
            \bigl(1-e^{-\underline\kappa\Delta}\bigr),
\end{gather*}
the flow Hessian satisfies for all \(t\ge1\)
\[
  \bigl\|\nabla^{2}\Phi^{t}(T_{0},\iota_{1}, \ldots, \iota_{t-1})\bigr\|_{\mathrm{op}}
  \le M_{\Delta}
  +\frac{M_{\Delta}\bigl(1+\tfrac{r}{K_{1}}\bigr)^{2}}%
        {1-e^{-K_{1}\Delta}}.
\]
\end{lemma}
\begin{proof}
    The proof is presented in Appendix \ref{app:proof_flow_map_1}.
\end{proof}
We now prove that this is a sufficient condition to recover quadratic descent, under mild conditions. 
\begin{proposition}[Descent in temperature update]%
\label{prop:sufficient_descent}
Let \(\Lambda>\sum_{i=1}^{|\mathcal{T}|}\lambda_{i}\) denote an upper
bound on the KKT multipliers of each $\mathrm{Tmp}_{ij}$ subproblem and choose
\(\rho^{(r-1)}\ge2C_{\Delta}\Lambda\).
If \(y^{(r)}\) is a stationary point for \(\mathcal{L}\), then
\begin{align*}
  \mathcal{L}(x^{(r)},y^{(r-1)},\ldots)
 -&\mathcal{L}(x^{(r)},y^{(r)},\ldots)
 \\&\ge\
 \frac{\rho^{(r-1)}}{4}
 \bigl\|B\,y^{(r)}-B\,y^{(r-1)}\bigr\|^{2}.
\end{align*}
\end{proposition}
\begin{proof}
    The proof is presented in Appendix \ref{app:sufficient_descent}.
\end{proof}
Finally, we prove the convergence of the ADMM algorithm:
\begin{proposition}[Inner convergence rate]%
\label{prop:inner_convergence}
Assume that for every inner iteration \(r\),
\(
  \mathcal{L}(x^{(r)},y^{(r-1)},\ldots)
 \le
  \mathcal{L}(x^{(r-1)},y^{(r-1)},\ldots),
\) and that $x^{(r)}$ is a stationary point of the Lagrangian. 
Then, under the hypothesis of Proposition~\ref{prop:sufficient_descent}, Algorithm~\ref{alg:inner-iteration} attains an $\epsilon$‑stationary point of \eqref{eq:compact_original_separable_problem} in \(O(1/\epsilon^{2})\) iterations.
\end{proposition}
\begin{proof}
    We satisfy the hypothesis of Theorem 3 in \cite{gholami_admm-based_2023}. Therefore, the conclusion holds. 
\end{proof}

%------------------------------------------------------------------------
\subsection{Global Convergence}\label{sec:global}
We now state Proposition~\ref{prop:outer_convergence}, which establishes convergence of the bi‑level ADMM algorithm for both transient and steady‑state line temperature dynamics.
\begin{proposition}[Outer convergence rate]%
\label{prop:outer_convergence}
Algorithm~\ref{alg:outer-iteration} reaches an $\epsilon$‑stationary point of \eqref{eq:compact_original_separable_problem} in at most \(O(1/\epsilon^{4})\) iterations. If the (unprojected) outer dual sequence is bounded (\(\|\hat w^{k}\|\le W\)), the rate improves to \(O(1/\epsilon^{3})\).
\end{proposition}
\begin{proof}
    We refer the reader to Theorem 2 of \cite{sun_two-level_2021} and Theorems 1-2 in \cite{sun_two-level_2023}.
\end{proof}
\section{Computations}\label{sec:computations}
\noindent
\begin{table*}[t]
\centering
\caption{Convergence behavior of ADMM in different ratings setups.}
\label{tab:convergence_stats}
\begin{tabular}{@{}l *{4}{S[table-format=4.2]@{\hskip 1mm}S[table-format=4.2]@{\hskip 1mm}S[table-format=4.2]}@{}}
\toprule
 & \multicolumn{3}{c}{Num. Outer Iter.} 
 & \multicolumn{3}{c}{Num. Inner Iter.} 
 & \multicolumn{3}{c}{Total time (s)} 
 & \multicolumn{3}{c}{$\norm{Ax + By}_\infty$ ($\cdot 10^{-3})$} \\
\cmidrule(lr){2-4} \cmidrule(lr){5-7} \cmidrule(lr){8-10} \cmidrule(lr){11-13}
Setup & {Min} & {Max} & {Mean} 
      & {Min} & {Max} & {Mean} 
      & {Min} & {Max} & {Mean} 
      & {Min} & {Max} & {Mean} \\
\midrule
DLR-SS & 1.00 & 12.00 & 8.13
     & 12.00 & 121.00 & 18.36
     & 59.82 & 1480.40 & 221.62
     & 0.00 & 6.66 &  \bfseries 2.52 \\
DLR-Trans & 1.00 & 13.00 & 8.14
    & 13.00 & 123.00 & 18.90
    & 66.92 & 1494.94 & 229.05
    & 0.00 & 12.60 & 2.99 \\
AAR & 1.00 & 13.00 & \bfseries 8.03
    & 1.00 & 125.00 & \bfseries 17.63
    & 49.73 & 1784.90 & \bfseries 216.47
    & 0.00 & 6.77 & \bfseries 2.52 \\
SLR & 1.00 & 14.00 & 8.14
    & 1.00 & 177.00 & 18.69
    & 45.49 & 1947.04 & 221.43
    & 0.00 & 8.31 & 2.56 \\
\bottomrule
\end{tabular}
\vspace{-1,5em}
\end{table*}

We now present the computational results, including the convergence behavior of our algorithm and comparisons between different rating schemes.

\subsection{Experimental Setup}
We compare four transmission rating schemes:

\begin{enumerate}
    \item \textit{DLR-Trans}: We solve the transient-state model given by~\eqref{eq:compact_original_separable_problem}, where the temperature variations are computed given the temperature model $\mathrm{Tmp}^{\mathrm{Trans}}$ (Eq.~\eqref{eq:reform}).
    \item \textit{DLR-SS}: We solve the model~\eqref{eq:compact_original_separable_problem}, but the temperature variations are given by the steady-state model~\eqref{eq:ss_model}. In this formulation, current limits are set a priori in the AC model, using Eq.~\eqref{eq:model:current_limit}. Therefore, the ADMM decomposition only enforces ramping constraints.
    \item \textit{AAR} (Ambient Adjusted Ratings): We solve~\eqref{eq:compact_original_separable_problem}, with steady-state dynamics, similarly to~\eqref{eq:ss_model}. However, the current ratings are computed using conservative values of wind speed and angles $(v_w = 0.6 m/s, \phi = \frac{\pi}{2})$, but real data for ambient temperature. 
    \item \textit{SLR}: We solve~\eqref{eq:compact_original_separable_problem}, but we set a priori current limits in the AC model corresponding to conservative weather conditions: wind data is the same as in AAR, while $T_a = 40^\circ C$ in the summer and $20^\circ C$ in the winter.
\end{enumerate}
All four models are solved with Algorithm \ref{alg:outer-iteration}. For \textit{DLR-Trans}, the algorithm enforces consensus on both current variables $\iota$ and power injections $p^G$. For all other setups, consensus is only over variables $p^G$, as current limits are already imposed in the ACOPF models (see Proposition \ref{prop:ss_restriction}). 

\subsection{Data used} \label{sec:data_used}
We evaluate our algorithm on the ERCOT 2000-bus grid from the TAMU dataset \cite{birchfield2016grid}, augmented with generator data from EIA Form-860 \cite{noauthor_annual_nodate}. Zonal demand time series (5-minute resolution) are sourced from the Grid Status API \cite{noauthor_grid_nodate}. NOAA HRRR forecast data \cite{noaa_hrrr}, accessed via Herbie \cite{blaylock2024herbie}, provides 15-minute weather inputs. Wind speeds at 80m and 10m are interpolated to 40m; $T_a$ and $K_{\text{angle}}$ values are sampled conservatively along transmission lines. Wind and solar capacity factors are computed using GE 2.75-120 curves \cite{noauthor_ge_nodate} and pvlib \cite{holmgren2018pvlib}, respectively. Thermal ramp rates are 20\% per 5 minutes; renewables are not ramp-limited. Renewable units are assigned a power factor of $|Q^{\max}|/P^{\max} \approx 0.329$. Unit commitment is excluded and $P^{\min} = 0$. Moreover, $T^{\max} = 100^\circ C$. 

\subsection{Implementation details} \label{sec:implementation_details}
All the optimization code is written in Julia 1.11.5 and runs on the MIT Engaging HPC system on 16 cores of an AMD EPYC 9474F 48-Core Processor. Nonlinear constrained problems are modeled in JuMP~\cite{dunning_jump_2017} and solved with IPOPT~\cite{wachter_implementation_2006} using the MA57 linear solver~\cite{duff_ma57---code_2004}.

All variables in the temperature models are rescaled so that their range is between $10^{-1}$ and $10^4$. Consensus in power injection variables is enforced in per-unit. Consensus in current magnitude squared is enforced in the unit of $(kA)^2$.

Let $d$ be the size of the vectors $Ax^k$. Typically $d = |\sT| \cdot (|\sG| + |\sE|)$. We set the parameters $\theta = 100$, $\gamma = 6.0$, $\omega = 0.6$, $\Delta = 300$ seconds, $|\sT| = 12$ for a $1$-hour horizon. 

% In the \textit{SLR}, \textit{AAR} and \textit{DLR-SS} setups, the ADMM algorithm alternates between solving in parallel ACOPF problems with current limits, as well as ramping blocks. In the \textit{DLR-Trans} setup, the algorithm also solves temperature blocks and enforces consensus on the generation power and the current magnitude. 

% Adaptive inner penalty parameter updates, such as H3 in \cite{gholami_admm-based_2023}, do not improve performance significantly but make the algorithm less stable. We do not use them. 

Throughout the experiments, we use the termination criterion $\epsilon = 10^{-4}$. We exit the inner loop (Algorithm \ref{alg:inner-iteration}) whenever the algorithm returns an iterate satisfying $\norm{Ax^{k} + By^{k} + u^{k}}_2 \leq \max(\frac{1}{k}\sqrt{\frac{d}{\theta^k}}, \epsilon)$.
The outer loop terminates when the algorithm returns a point satisfying $\norm{Ax^{k} + By^{k}}_2 \leq \sqrt{d} \epsilon$. 

\subsection{Empirical convergence of the ADMM algorithm} \label{sec:empirical_convergence}

\begin{figure*}[t]
    \centering
    \hfill
     \subfloat[]{%\
     \label{fig:consensus_gap_consensus}
        \includegraphics[width=0.25\linewidth]{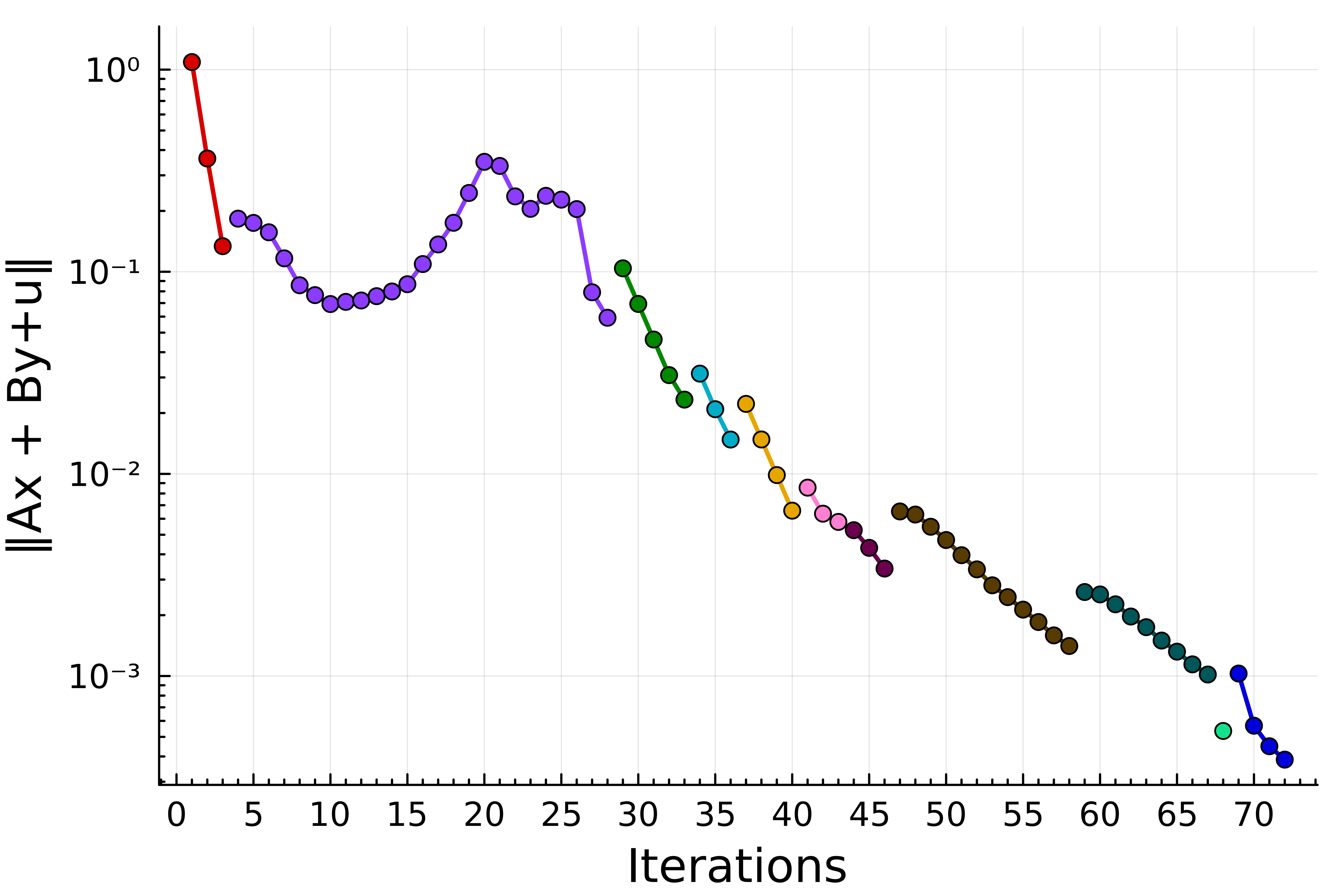}
        }\hfill
        \subfloat[]{%
             \label{fig:consensus_gap_gap}
        \includegraphics[width=0.25\linewidth]{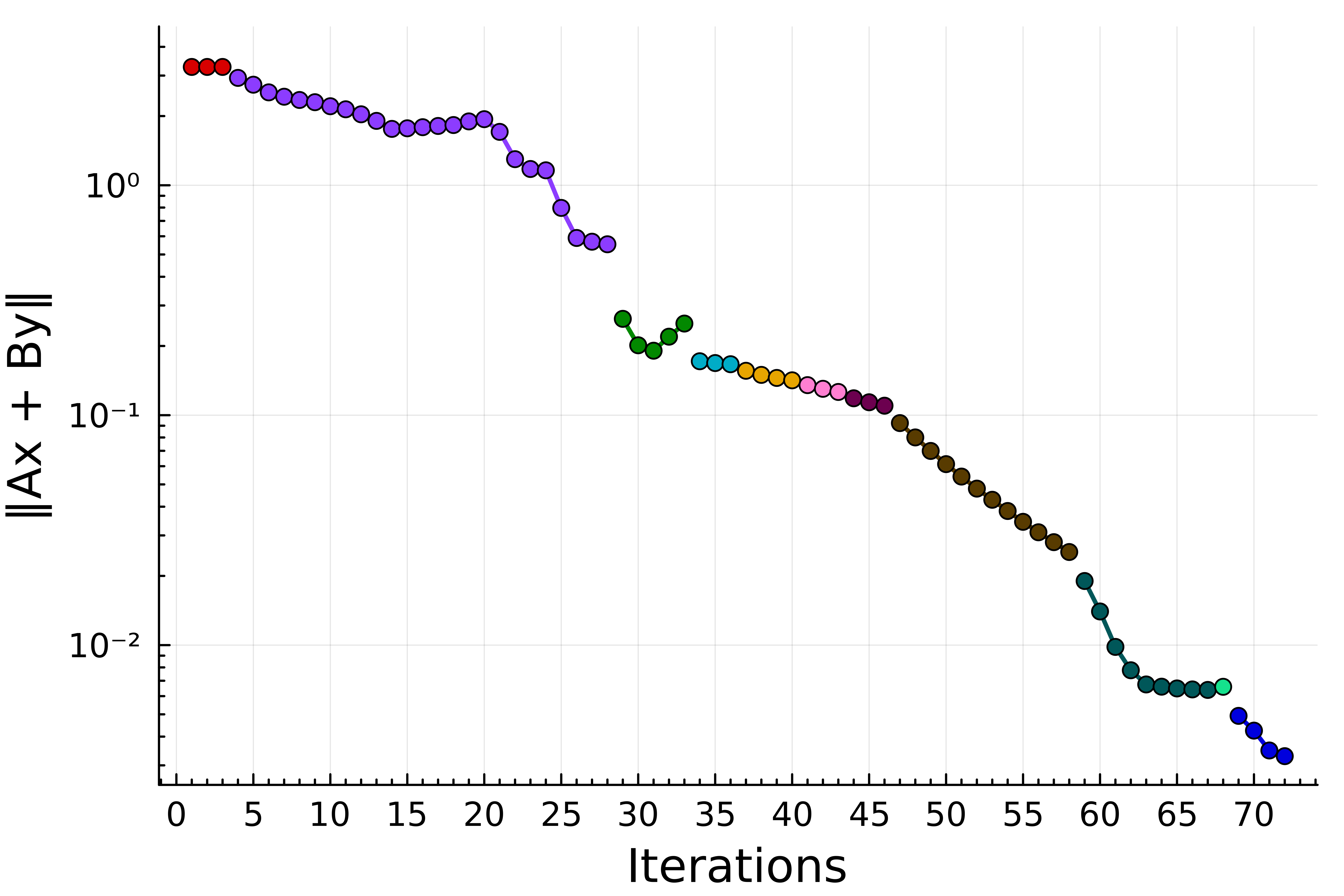}
        }\hfill
        \vspace{-0.5em}
    \caption{Convergence behavior of the ADMM algorithm showing (a) inner iteration consensus and (b) primal feasibility gap. Colors denote outer iterations.}
    \label{fig:convergence_and_primal_gap}
    \vspace{-1.5em}
\end{figure*}

Table~\ref{tab:convergence_stats} reports convergence results over a week of 5-minute operations (July 1–14, 2024), including the number of outer/inner iterations, total solve time, and final primal feasibility gap. The bi-level ADMM algorithm converged across all configurations and hours, finding a primal feasible solution within the prescribed tolerance.

\begin{figure*}[b]
\captionsetup{farskip=0pt,nearskip=0pt}
  \centering
  \subfloat[]{%
    \includegraphics[width=0.3\linewidth]{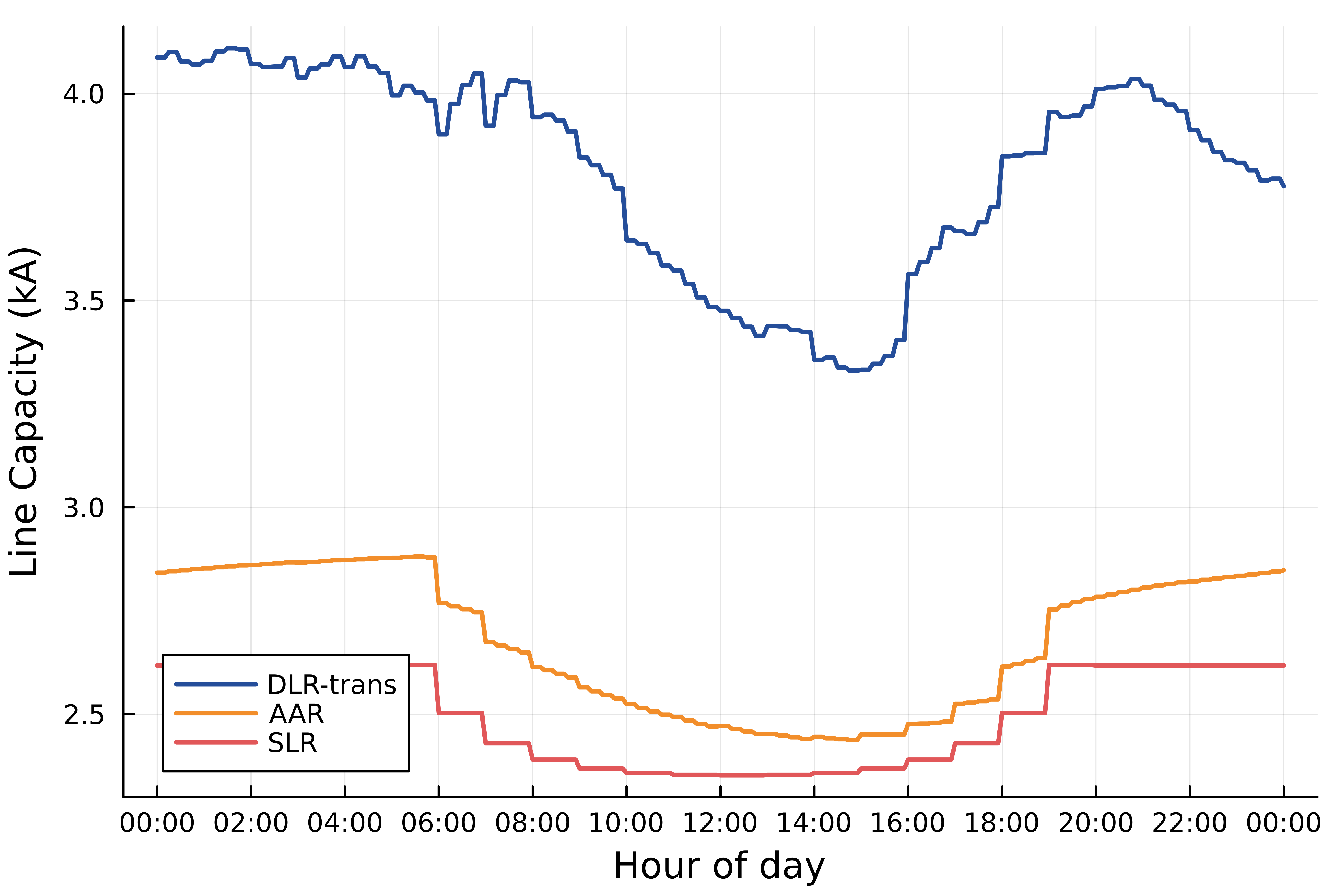}%
    \label{fig:linecap}%
  } \hfill
  \subfloat[]{%
    \includegraphics[width=0.3\linewidth]{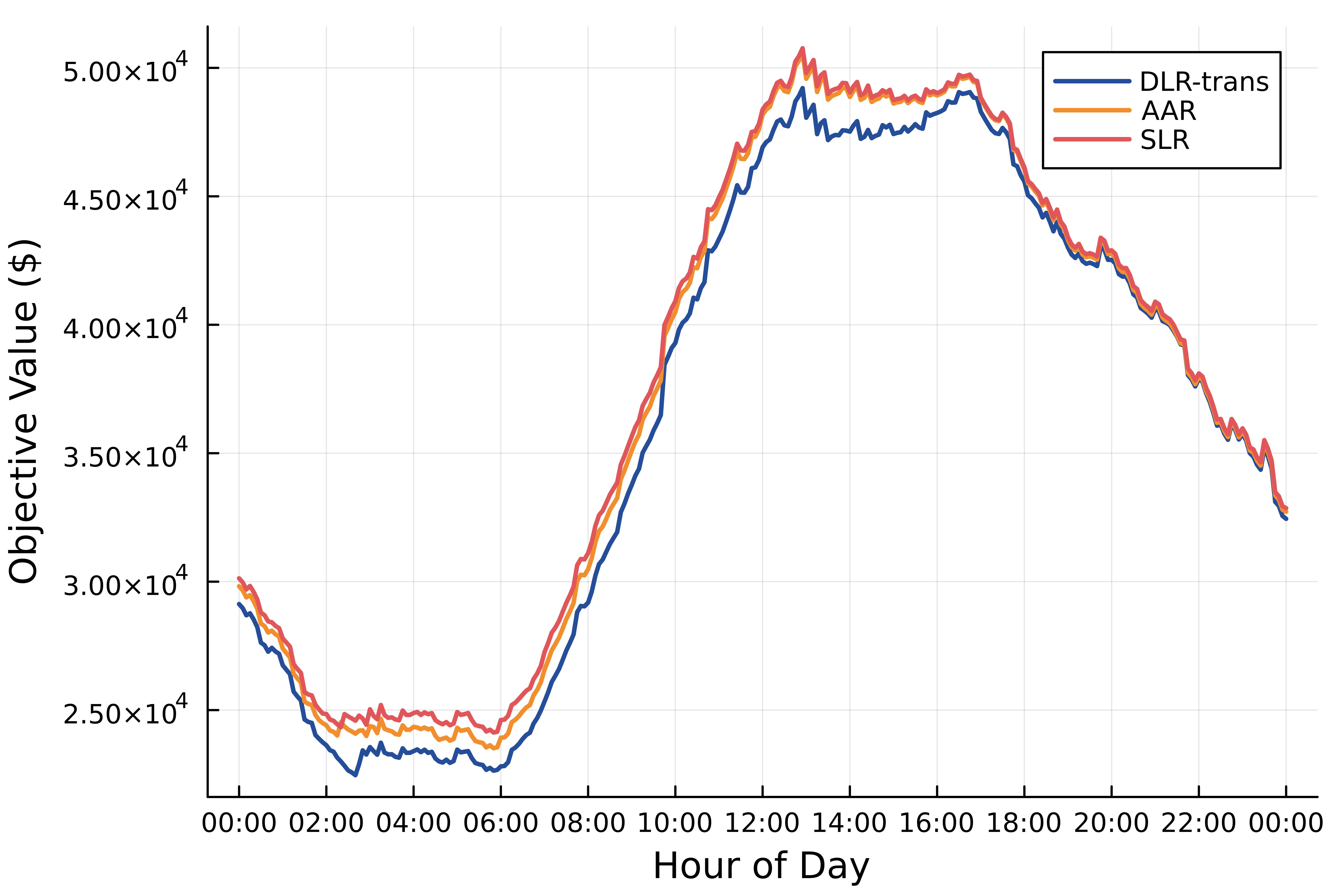}%
    \label{fig:obj}%
  }\hfill
  \subfloat[]{%
    \includegraphics[width=0.3\linewidth]{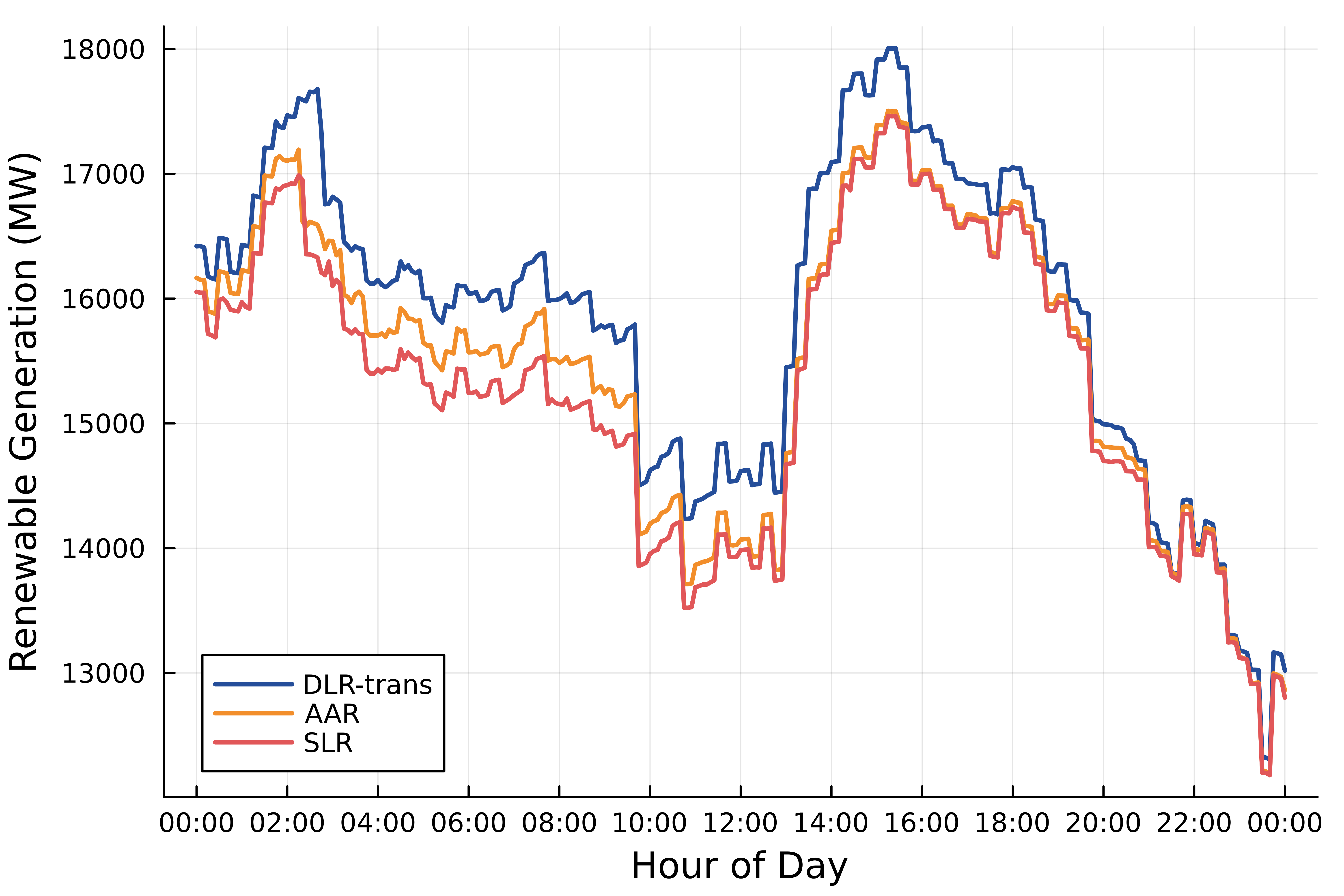}%
    \label{fig:ren_gen}%
  }
  \vspace{-0.5em}
  \caption{Line capacity (a), system cost (b), and (c) renewable output of all rating schemes over one summer day of operations.}
  \vspace{-1em}
  \label{fig:performance_one_day}
\end{figure*}
Our preliminary experiments show that Algorithm \ref{alg:outer-iteration} converges faster and to better solutions when its initial iterates are already nearly feasible for Problem~\eqref{eq:compact_original_separable_problem}. Therefore, it is detrimental for overall performance and speed to remove all line limits \eqref{eq:Imax} constraints when solving \textit{DLR-Trans}. Instead, we identify candidate lines via a screening procedure: we first solve $|\sT|$ instances of the ACOPF using steady-state line limits (Eq.~\eqref{eq:model:current_limit}) on all lines. Next, we identify lines whose initial temperature is below $90^\circ\text{C}$ and that reach $100^\circ\text{C}$ during some period $t \in \sT$. For these lines only, we relax the current limit and apply the transient-state temperature model within our ADMM framework. For all other lines, we keep the steady-state current limits given by Eq.~\eqref{eq:model:current_limit}. This applies the transient model only to lines that effectively provide additional capacity. 
It is worth noting that for all models, almost all of the computational time is spent solving the ACOPF subproblems. We observe that all models using the steady-state computations of temperature exhibit similar computational complexity. One may think that the \textit{SLR} models would be easier to solve, because they are conceptually simpler. However, for \textit{AAR}, \textit{SLR}, and \textit{DLR-SS}, the current limits are precomputed via~\eqref{eq:Imax}. Once the weather data has been gathered, all models share the same optimization formulation, making \textit{DLR} or \textit{AAR} not harder to solve. The average runtimes of \textit{AAR}, \textit{SLR}, and \textit{DLR-SS} are similar, but \textit{AAR} and \textit{SLR} have a higher maximal computational time. In heavily loaded conditions, the tighter current limits make solving ACOPF subproblems harder. 
The computational performance of \textit{DLR-Trans} is also comparable. This is because only a small subset of lines is evaluated dynamically. 

\begin{table}[t]
  \caption{Performance of rating methods, relative to SLR baseline}
  \label{tab:methods}
  \centering
  \begin{tabular*}{\columnwidth}{@{\extracolsep{\fill}}l@{\hskip 1pt}cc@{\hskip 1pt}cc@{\hskip 1pt}cc}
    \toprule
    \multirow{2}{*}{\textbf{Method}} &
    \multicolumn{2}{c}{\textbf{Capacity\ (\%)}} &
    \multicolumn{2}{c}{\textbf{Cost\ (\%)}} &
    \multicolumn{2}{c}{\textbf{RE Gen.\ (\%)}} \\
    \cmidrule{2-7}
     & Summer & Winter & Summer & Winter & Summer & Winter \\
    \midrule
    \textit{SLR}     & 0.00 & 0.00 & 0.00 & 0.00 & 0.00 & 0.00 \\
    \textit{AAR}     & 8.72 & 6.24 & -0.87 & -0.72 & 1.24 & 0.51 \\
    \textit{DLR-SS}  & \textbf{43.30} & \textbf{52.10} & -2.11 & \textbf{-1.81} & 3.10 & \textbf{1.32} \\
    \textit{DLR-Trans} & \textbf{43.30} & \textbf{52.10} & \textbf{-2.12} & \textbf{-1.81} & \textbf{3.11} & \textbf{1.32} \\
    \bottomrule
  \end{tabular*}
  \vspace{-1em}
\end{table}

% Additionally, when using the dynamic model \textit{DLR-Trans} across all lines, the ADMM algorithm tends to penalize consensus on variables that are not constrained by temperature, which degrades solution quality during convergence. Therefore, at the start of the solve, we screen the lines and apply the transient state temperature model only where it is needed.
% Based on Proposition \ref{prop:ss_restriction}, the transient-state model provides additional thermal margin under the following conditions:
% \begin{enumerate}
% \item \textbf{Heating}: The line temperature increases over time.
% \item \textbf{Congestion}: The line reaches its maximum allowable temperature.
% \end{enumerate}

Figure~\ref{fig:consensus_gap_consensus} shows that consensus decreases by several orders of magnitude per iteration, approximating a stationary point to~\eqref{eq:compact_augmented_lagrangian}. Subsequent iterations enforce a stronger consensus to achieve feasibility.
In Figure \ref{fig:consensus_gap_gap}, we observe that in the first iterations, the feasibility gap remains relatively large. As the number of outer iterations and the penalty parameter increase, the primal gap reduces. In the end, we see a sharp reduction in the gap, triggering termination of the algorithm.

\subsection{Congestion Alleviation of DLR} \label{subsec:congestion_alleviation}

% We now present the experimental results comparing the four models over the period from July $1^{\text{st}}$, 2024 to July $8^{\text{th}}$, 2024.

Table~\ref{tab:methods} reports relative variations of average current capacity across all lines, system cost, and total renewable generation, measured against the \textit{SLR} baseline. Current capacity is computed using the steady-state maximum current (Eq.~\eqref{eq:Imax}) and is therefore identical for both \textit{DLR-SS} and \textit{DLR-Trans}. Summer represents runs from July 1 to July 14, 2024, whereas winter represents runs from January 1 to January 14, 2024
% Since the TAMU dataset has a peak load snapshot, the commitment is more representative of summer days, and the 

We observe a significant average current capacity increase in both the summer and winter DLR runs. The increase is even higher in the winter runs, when the wind is stronger. These variations in current capacity translate into cost reduction in both winter and summer days. On average, DLR reduces system cost by around $2\%$, and boosts renewable generation by $1{-}3\%$, thereby improving upon \textit{AAR} $4$ times for capacity and $2.5$ times for cost reduction. In the summer, during high congestion scenarios, \textit{DLR-Trans} shows marginal improvement over \textit{DLR-SS}. In winter, when congestion is low, \textit{DLR-SS} and \textit{DLR-Trans} performance is comparable. For the most part, system-wide cost benefits are captured by the steady-state model.

Figure~\ref{fig:performance_one_day} shows current capacity, system cost, and renewable generation for all models on July 4, 2024. The benefits of DLR methods are visible, with consistent improvements over both \textit{AAR} and \textit{SLR}. Because of its overall similarity to \textit{DLR-Trans}, \textit{DLR-SS} was omitted.  
%Minor, localized differences can be observed between \textit{DLR-SS} and \textit{DLR-Trans}, though overall performance remains comparable.

\subsection{\textit{DLR-SS} vs \textit{DLR-Trans} in transient regimes} \label{sec:ss_dyn}
Despite their similar cost performance, \textit{DLR-SS} and \textit{DLR-Trans} yield different solutions.
In Figure~\ref{fig:dyn_current_increase}, we focus on July 6, 2024, highlighting lines that are congested under \textit{DLR-SS} and show significantly different current dispatch under \textit{DLR-Trans}. On these lines, \textit{DLR-Trans} enables up to $15$\% more current during certain periods (e.g. line 1). It can be observed that increasing the current on a congested line during a time period sometimes reduces the current supplied on other congested lines, effectively adding headroom to these lines. 

This demonstrates that the transient state framework provides operators with greater flexibility to respond to rapid and localized weather changes. These advantages last for $5$ to $20$ minutes, before the transient temperatures converge to their steady-state.

Transient conditions usually involve only a few lines at any given moment, so the benefits of \textit{DLR-Trans} are both brief and highly localized. Across the entire network, $1.5\%$ of the lines that are congested under \textit{DLR-SS} gain extra capacity under \textit{DLR-Trans}, and that happens in $5\%$ of the time intervals. The additional transient capacity provides operators with a powerful tool for managing short-term stress on the grid. Even brief, localized flexibility can be critical in maintaining reliability under rapidly changing conditions.
% The elevated power flows on congested lines represent an average net increase of 0.08\% in total current during periods of congestion.

\begin{figure}[h]
    \centering
    \includegraphics[width=0.7\linewidth]{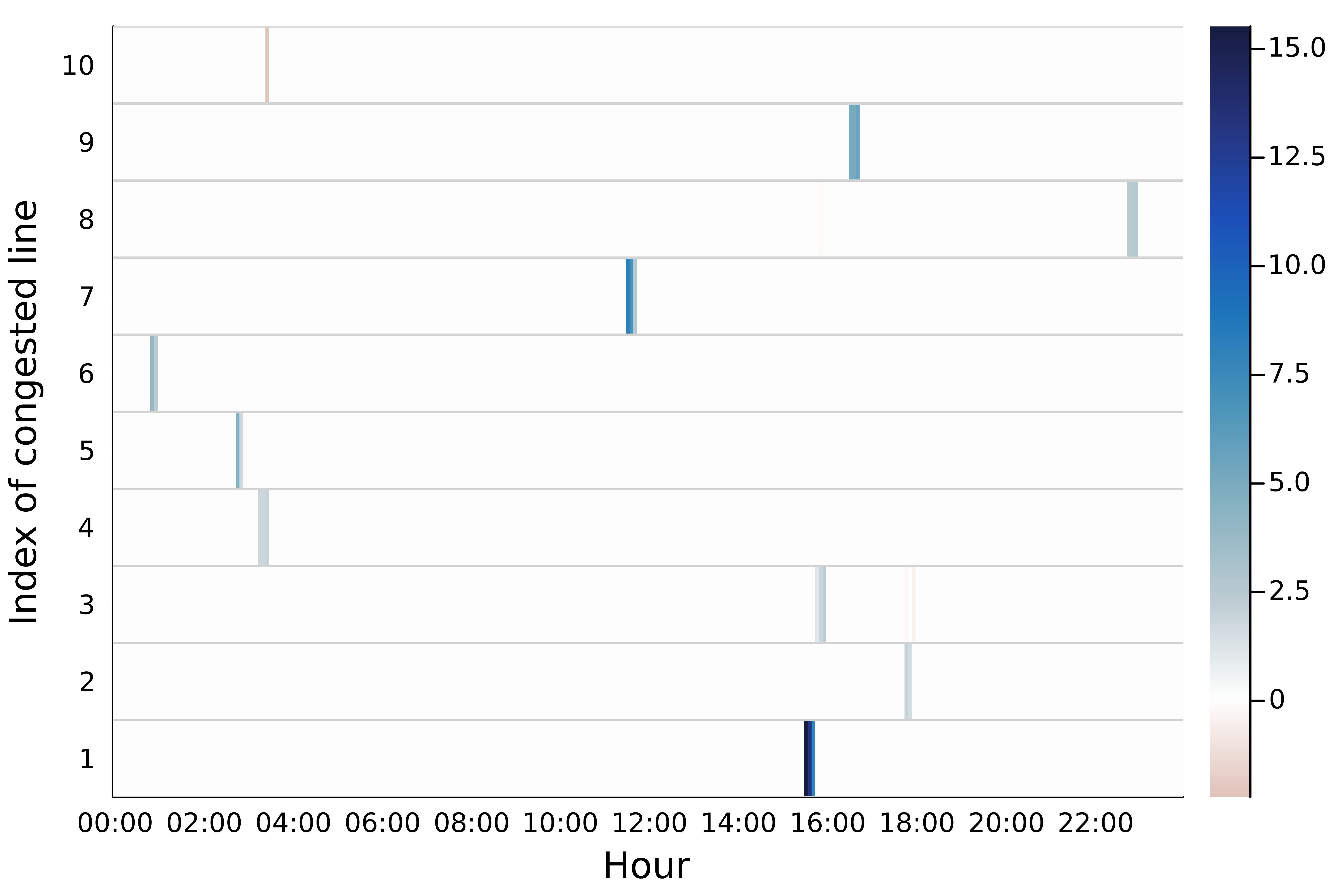}
    \vspace{-1em}
    \caption{Current difference (in \%) between \textit{DLR-Trans} and \textit{DLR-SS}. The selected lines are those that are congested in \textit{DLR-SS} and in which the current supplied is different in \textit{DLR-Trans}.}
    \label{fig:dyn_current_increase}
    \vspace{-1em}
\end{figure}
\section{Conclusions}
\noindent
In this paper, the transient-state DLR-ACOPF model was introduced and solved using a bi-level ADMM algorithm. 
The model explicitly incorporates the heat equation of transmission lines and leverages a space-time decomposition approach. Through large-scale computational experiments, we demonstrated that using DLR at the grid level can reduce generation costs by over 2\% under typical summer operating conditions, compared to the baseline SLR scenario. Furthermore, the transient-state temperature computations offer substantial headroom benefits with minimal additional computational cost relative to the steady-state approach.

% \section*{Acknowledgments}
% The authors acknowledge the MIT Office of Research Computing and Data for providing high-performance computing resources that have contributed to the research results reported within this paper.

% \FloatBarrier
\appendices{}

\section{Proof of Theorem \ref{thm:tempsol}}\label{app:proof_tempsol}

\begin{proof}
    The temperature dynamics~\eqref{eq:temp-dyn-linear-qc} can be written in a separable form as
        \begin{align}
        -K_4 {d\tau} = \frac{dT}{T^4 + (K_1/K_4) T - (K_0/K_4)} =:\frac{dT}{P(T)}. \label{eq:dTdt-separable}
    \end{align}
    Define $b=K_1/K_4, c=K_0/K_4$. Using the discriminant of $P(T)$ given as $\Delta = -256c^3 - 27b^4$, and $c<0$, we know that $f(T)=0$ has two real roots and two complex roots.
    Moreover, the two real roots are given by the intersection of the quartic curve $T^4$ and the line $-bT+c$. Thus, one of two real roots is positive, denoted as $s_1$, and the other is negative, denoted as $-s_2$. Factorizing $f$, we get that:
        $P(T) = (T - s_1)(T + s_2)(T^2 - pT + q)$
     with 
        $p = s_2 - s_1$ and 
        $q = s_1 s_2 + (s_2 - s_1)^2 > 0.$

    Let us define $A, B, C, D$ such that: $\frac{1}{P(T)} = \frac{A}{T-s_1} + \frac{C}{T+s_2} + \frac{BT+D}{T^2 - pT + q}.$ Standard algebra gives us: $ A = \frac{3s_2^2 - 2 s_1 s_2 + s_1^2}{(s_1 + s_2)g(s_1, s_2)}$, $C = \frac{-(3s_1^2 - 2 s_1 s_2 + s_2^2)}{(s_1 + s_2)g(s_1, s_2)}$, $B = -A - C$, $D =  \frac{s_1^2 - 4s_1s_2 + s_2^2}{g(s_1, s_2)}$, with $g(s_1, s_2) := (3s_2^2 - 2 s_1 s_2 + s_1^2)(3s_1^2 - 2 s_1 s_2 + s_2^2)$.
    We integrate between $T_0$ and $T$ to obtain the claimed result. \end{proof}
    % \begin{align*}
    %    K_4 g(r_1,r_2) \tau &= -\frac{r_1^2-2r_1r_2+3r_2^2}{r_1+r_2}\log\frac{|T-r_1|}{|T_0-r_1|}\\ &\quad + 
    %    \frac{3r_1^2 - 2 r_1 r_2 + r_2^2}{r_1 + r_2} \log\frac{|T + r_2|}{|T_0 + r_2|}\notag \\
    %    & \quad + (r_2-r_1)\log\frac{|T^2 -pT + q|}{|T_0^2 - pT_0 +q|} \\ & \quad +\frac{4r_1r_2}{\sqrt{4q-p^2}}\left(\phi(T) - \phi(T_0) \right),
    %  \end{align*} 
    %  with $\phi(T):= \arctan\frac{2T-p}{\sqrt{4q-p^2}} $.

\allowdisplaybreaks
\section{Proofs of section \ref{sec:convergence}}
\subsection{Proof of Lemma~\ref{thm:Hessian}}\label{app:proof_flow_map_1}
\begin{proof}
We first establish a single‑period estimate.

\begin{lemma}\label{lem:single-period}
With \(\beta\) and \(M_{\Delta}\) as above,
\[
\bigl\|\nabla^{2}f_{W, \tau}(T_{0},\iota_{1})\bigr\|_{\mathrm{op}}
\;\le\;
M_{\tau}\qquad(0\le\tau\le\Delta).
\]
\end{lemma}

\begin{proof}[Proof of Lemma~\ref{lem:single-period}]
Denote \(J=\partial_{x_{0}}f_{W, \tau}\) and \(H=\partial^{2}_{x_{0}}f_{W, \tau}\).
The gradient satisfies the ODE:
\[
    \dot J= P'(T)J+
    \begin{pmatrix}0\\ r\end{pmatrix},
    \qquad
    J(0)=\begin{pmatrix}1\\ 0\end{pmatrix}
\]
Since \(-K_{1}\ge P'(T)\ge -\underline\kappa:=-K_{1}-4K_{4}(T^{\max})^3\),
Grönwall gives \(e^{-\underline\kappa\tau}\le J_{1}\le e^{-K_{1}\tau}\) and \(r\bigl(1-e^{-\tau\underline \kappa}\bigr)/K_{1}\le J_{2}\le  r\bigl(1-e^{-K_{1}\tau}\bigr)/K_{1}\).
The Hessian obeys
\[
    \dot H=P'(T)H+P''(T)JJ^{\!\top},\qquad H(0)=0_{2},
\]
and \(P''(T)\in[-\beta,0]\). By variation of the constant, we get
\(
    H(\tau)=\int_{0}^{\tau}\Psi(\tau, s)
         P''(T)J(s)J(s)^{\top}ds
\) with \(\Psi(s, \tau) := \exp(\int_s^\tau P(T(u))du)\). Since the matrix $JJ^{\top}$ is rank one, we have:
\(
    \|H(\tau)\|_{\mathrm{op}}
    \le\beta\int_{0}^{\tau} \Psi(s, \tau) \norm{J(s)}_2^2ds
    \le \frac{\beta}{\kappa}(1 + \frac{r^2}{K_1^2})(1 - e^{-\underline \kappa \tau}) \leq M_\tau
\)
\end{proof}
\noindent
Returning to the theorem, write the chain‑rule in first order for the function $\Phi^t(\iota_1, \ldots, \iota_T)$, for a given initial $T_0$ and $t \geq 2$:
\[
    g_t :=\partial_{\iota_1, \iota_{t-1}} \Phi^t = J_2 e_t + J_1 \partial_{\iota_1, \iota_{t-2}}\Phi^{t-1}
\]

As a result, we have that $\norm{\partial_{\iota_1, \iota_{t-1}} \Phi^t} \leq \frac{r}{K_1}$, for all $t$.

We now use the chain rule for the Hessian of $\Phi^{t}$. Define $h_t(\iota_{1:t-1}) = f_{W, \Delta}(T_{t-1}(\iota_{1:T-2}), \iota_{t-1})$. We have:

\[
H_t =J_{h_t}^{\top}\bigl(\nabla^{2}f_{W, \Delta}\bigr)J_{h_t} + J_1 H_{t-1},
\]
where $J_{h_t} := (g_{t-1}, e_t)^\top$. 

This gives $\norm{H_t}_{op} \leq e^{-K_1 \Delta}\norm{H_{t-1}}_{op} + (1 + \norm{g_{t-1}})^2M_{\Delta}$. Solving the recursion yields:
\begin{align}
    \norm{H_t}_{op} \leq M_\Delta + \frac{M_\Delta(1 + \frac{r}{K_1})^2}{1 - e^{-K_1 \Delta}}.
\end{align}
This concludes the proof.
\end{proof}
\subsection{Proof of Proposition \ref{prop:sufficient_descent}} \label{app:sufficient_descent}
\noindent
\textit{Proof.}
    Let $y$ denote a solution of $\mathrm{Tmp}_{ij}$, for any $ij \in \sE$.
    For every feasible point $y$, the Jacobian
    \(
    \bigl[\nabla\Phi^{t}(y)\bigr]_{t\in\mathfrak T}
    \)
    is lower–triangular with non‑zero diagonal, hence its columns are linearly
    independent. LICQ therefore holds \cite{solodov_constraint_2011}, so every local minimizer satisfies the Karush–Kuhn–Tucker (KKT) conditions.
    % Moreover, the Jacobian matrix has a minimal eigenvalue of $\lambda_{\min} := \frac{r}{K_1}(1 - e^{\Delta \underline \kappa})$, using the notations of Theorem~\ref{thm:Hessian}.

    Let $g(y):=
    -B^\top \left(v^{(r-1)}-\rho^{(r)}\!\bigl(Ax^{(r)}+By^{(r)}+u^{(r-1)}\bigr)\right)$, and $\mathfrak T=\{t:\Phi^{t}(y)=T^{\max}\}$ and let $\lambda_t\ge 0$ ($t\in\mathfrak T$) be the dual multipliers.
    Stationarity gives
    \begin{equation}\label{eq:kkt}
    g(y)+\sum_{t\in\mathfrak T}\lambda_t\nabla\Phi^{t}(y)=0. 
    \end{equation}
    Because each $\Phi^{t}$ is $C_\Delta$‑smooth, for any $y$ satisfying $\Phi^{t}(y) = T^{\max}$, any other feasible $z$ satisfies
    $
    0\ge
    \Phi^{t}(z)-\Phi^{t}(y)
    \ge
    \nabla\Phi^{t}(y)^{\!\top}(z-y)-\tfrac{C_\Delta}{2}\|z-y\|_2^{2},
    \, \forall t\in\mathfrak T.$
    Multiplying by $\lambda_t$ and summing yields
    \begin{align*}
    g(y)^{\!\top}(z-y)
    &=-\!\sum_{t\in\mathfrak T}
      \lambda_t\nabla\Phi^{t}(y)^{\!\top}(z-y)
    \ge
    -\Lambda\,\frac{C_\Delta}{2}\|z-y\|_2^{2},  \notag
    \end{align*}
    since $\Lambda\geq\sum_{t\in\mathfrak T}\lambda_t$ by hypothesis. 
    Moreover, by first order optimality conditions of $\mathrm{Rmp}$, for any $y$ stationary point of $\mathrm{Rmp}$, any $z \in \mathrm{Rmp}$ we have $g(y)^\top (z-y) \geq 0 \geq - \Lambda \frac{C_{\Delta}}{2} \norm{z-y}_2^2$.

    As a result, for any $y$ stationary point of $\mathrm{Tmp} \times \mathrm{Rmp}$, for any $z$ of $\mathrm{Tmp} \times \mathrm{Rmp}$, we have:
    \begin{align}\label{eq:quad_descent}
        \bigl({-}{B^\top} v^{(r-1)} {-}B^\top\rho^{(r)}&(Ax^{(r)} {+} By^{(r)} {+} u^{(r-1)})\bigr)^\top (Bz{-}By) \notag \\ &\quad \geq- \Lambda \frac{C_{\Delta}}{2} \norm{Bz-By}_2^2.
    \end{align}

\noindent Using the identity \(\|a+b\|_2^{2}-\|a+c\|_2^{2}=2(a+c)^{\!\top}(b-c)+\|b-c\|_2^{2}\) with \(a=Ax^{(k,r)}+u^{(k,r-1)},\; b=By^{(k,r)},\; c=By^{(k,r-1)}\)
and inequality~\eqref{eq:quad_descent} we obtain
\begin{align*}\label{eq:L-descent}
    \mathcal L\bigl(\ldots,y^{(k,r-1)},v^{(k,r-1)}\bigr)
    -\mathcal L\bigl(\ldots,y^{(k,r)},v^{(k,r-1)}\bigr) \notag
    \\ \quad \ge
    \frac{\rho^{(k,r-1)}}{4}
    \bigl\|y^{(k,r)}-y^{(k,r-1)}\bigr\|_2^{2}. \qed
\end{align*}

\bibliographystyle{IEEEtran}
\bibliography{ieee_bib}

\end{document}